\documentclass[12pt,twosides]{amsart}
\usepackage{amssymb,amsmath,amsthm, amscd, enumerate, mathrsfs}
\usepackage{graphicx, hhline}
\usepackage[all]{xy}
\usepackage{braket}
\usepackage[usenames]{color}
\usepackage{hyperref}
\usepackage{fancyhdr}
\usepackage[top=30truemm,bottom=30truemm,left=30truemm,right=30truemm]{geometry}

\hypersetup{colorlinks=true}

\title{Non-vanishing theorem for generalized log canonical pairs with a polarization}
\author{Kenta Hashizume}
\date{2020/12/30}
\address{Graduate School of Mathematical Sciences, 
The University of Tokyo, 3-8-1 Komaba Meguro-ku Tokyo 153-8914, Japan}
\email{hkenta@ms.u-tokyo.ac.jp}

\newtheorem{thm}{Theorem}[section]

\newtheorem{lem}[thm]{Lemma}

\newtheorem{prop}[thm]{Proposition}

\newtheorem{ques}[thm]{Question}

\theoremstyle{definition}
\newtheorem{defn}[thm]{Definition}
\newtheorem{rem}[thm]{Remark}
\newtheorem{note}[thm]{Notation}

\newtheorem*{ack}{Acknowledgments} 
 
\newtheorem*{b-divisor}{b-divisors} 
\newtheorem*{sing}{Singularities of pairs} 
\newtheorem*{g-pair}{Generalized pairs} 
\newtheorem*{adj-g-pair}{Divisorial adjunction for generalized pairs} 
\newtheorem*{mmp-g-pair}{MMP for generalized pairs}

\newtheorem{step2}{Step}
\newtheorem{step3}{Step}
\newtheorem{step4}{Step}
\newtheorem{step5}{Step}

\newtheorem*{claim*}{Claim}
\begin{document}

\maketitle

\begin{abstract}
We prove that the non-vanishing conjecture holds for generalized lc pairs with a polarization. 
\end{abstract}

\tableofcontents

\section{Introduction}
We will work over an algebraically closed field of characteristic zero. 

In this paper, we deal with generalized pairs, which is an extended notion of log pairs. 
The notion was introduced by Birkar--Zhang \cite{bz} to prove the effectivity of the Iitaka fibration in some sense. 
In \cite{bz}, Birkar and Zhang proved effectivity and boundedness of several invariants, called the effective birationality, the ACC for generalized lc thresholds, and the global ACC. 
These results played critical roles to prove the main result of \cite{bz}. 
After that, applying them, Birkar proved the boundedness of complements \cite{birkar-compl} and the boundedness of Fano varieties with a fixed dimension and mild singularity \cite{birkar-bab}, which led us to a significant development in  birational geometry. 
Currently, generalized pairs are also used as powerful tools to prove results of  log pairs (\cite{moraga-mmp}, \cite{lt}, \cite{birkar-connectedness}, \cite{filipazzi-svaldi}).   

Generalized pairs are also interesting objects to study their geometry. 
The structure of generalized pairs naturally appears in the base variety of lc-trivial fibrations or Iitaka fibrations, higher codimensional adjunction formulas, and cones over lc varieties with nef anti-canonical divisors. 
A lot of results in birational geometry, in particular the minimal model program and the canonical bundle formula, is extended to the context of generalized pairs (\cite{hacon-moraga}, \cite{filcanbundleformula}, \cite{hanli}, \cite{hanliu-nonvanish}, \cite{hanliu-can-bundle-formula}, \cite{liu-sarkisov}). 
For the current status and open problems, see \cite{birkar-gen-pair} by Birkar. 

In this paper, we study the non-vanishing conjecture for generalized lc pairs. 
In the case of log pairs, the non-vanishing conjecture for lc pairs predicts that pseudo-effective log canonical divisor of an lc pair is effective up to $\mathbb{R}$-linear equivalence. 
The conjecture is one of the most important open problems in minimal model theory. 
By a similar way to the case of log pairs, we can define generalized lc pairs, generalized klt pairs, and so on, then we can formulate generalized lc pairs analogue of the non-vanishing conjecture. 
However, unfortunately, a generalized lc pair on an elliptic curve is a counterexample to the non-vanishing conjecture for generalized lc pairs. 
Hence, a weaker version of the conjecture, called weak non-vanishing conjecture (cf.~\cite[Question 3.5]{bhzariski}), is expected to hold. 
In addition to it, in general, the class of generalized lc pairs are strictly larger than the class of lc pairs. 
More precisely, there exists a normal projective variety having a structure of generalized lc pair but not having any boundary divisor with which the log pair is lc (\cite[Remark 4.13 (1)]{has-class}). 
So the weak non-vanishing conjecture for generalized lc pairs does not immediately follow from the non-vanishing conjecture for lc pairs. 


Though we have the above obstructions, the following main result of this paper shows that the non-vanishing conjecture for generalized lc pairs holds true when we consider generalized lc pairs with a polarization.

\begin{thm}\label{thm--abund-main-2}
Let $(X,B,\boldsymbol{\rm M})$ be a projective generalized lc pair such that $\boldsymbol{\rm M}$ is a finite $\mathbb{R}_{>0}$-linear combination of b-nef $\mathbb{Q}$-b-Cartier $\mathbb{Q}$-b-divisors. 
Let $A$ be an ample $\mathbb{R}$-divisor on $X$. 
Suppose that $K_{X}+B+A+\boldsymbol{\rm M}_{X}$ is pseudo-effective. 
Then followings hold true.
\begin{enumerate}[(1)]
\item \label{thm--abund-main-2-(1)}
There exists an effective $\mathbb{R}$-divisor $D$ on $X$ such that $K_{X}+B+A+\boldsymbol{\rm M}_{X} \sim_{\mathbb{R}}D$. 
\item \label{thm--abund-main-2-(2)}
Suppose further that $\boldsymbol{\rm M}_{X}$ is $\mathbb{R}$-Cartier. 
Then, for every real number $\alpha \geq 1$, there exists an effective $\mathbb{R}$-divisor $D_{\alpha}$ on $X$ such that $K_{X}+B+A+\alpha \boldsymbol{\rm M}_{X} \sim_{\mathbb{R}}D_{\alpha}$.
\end{enumerate}
\end{thm}

Note that the generalized pair $(X,B,\alpha \boldsymbol{\rm M})$ in Theorem \ref{thm--abund-main-2} (\ref{thm--abund-main-2-(2)}) may not be generalized lc. 
So Theorem \ref{thm--abund-main-2} (\ref{thm--abund-main-2-(2)}) is not a direct consequence of Theorem \ref{thm--abund-main-2} (\ref{thm--abund-main-2-(1)}). 
In the generalized klt case, we can reduce the theorem to the case of log pairs which follows from \cite{bchm}. 
When $\boldsymbol{\rm M}=0$, Theorem \ref{thm--abund-main-2} is the non-vanishing theorem for lc pairs with a polarization, which was proved in \cite{hashizumehu}. 
Because of gaps between generalized lc pairs and lc pairs, Theorem \ref{thm--abund-main-2} does not directly follow from \cite{hashizumehu}. 

We can consider the same assertion as in Theorem \ref{thm--abund-main-2} (\ref{thm--abund-main-2-(2)}) without the $\mathbb{R}$-Cartier property of $\boldsymbol{\rm M}_{X}$ because we need not $\mathbb{R}$-Cartier property to define $\mathbb{R}$-linear equivalence.  
It is not clear that the $\mathbb{R}$-Cartier property of $\boldsymbol{\rm M}_{X}$ in Theorem \ref{thm--abund-main-2} (\ref{thm--abund-main-2-(2)}) can be removed, but a partial affirmative answer to the question will be given by Theorem \ref{thm--non-vanishing-lc-divisor-2}. See Section \ref{sec5} for the question of this direction and a related topic.  
As a special case of Theorem \ref{thm--non-vanishing-lc-divisor-2} we introduce a result for projective generalized lc pairs $(X,B,\boldsymbol{\rm M})$ with a polarization $A$ satisfying ``pseudo-effectivity of log canonical part $K_{X}+B+A$'' in some sense. 
In \cite{lazic-peternell-gen-abund-I} and \cite{lazic-peternell-gen-abund-II}, Lazi\'c and Peternell discussed on generalized klt pairs (without polarizations) under the assumption on the pseudo-effectivity of the log canonical part. 
Especially, in \cite{lazic-peternell-gen-abund-I}, they proved connections between conjectures of minimal model theory for klt pairs and those for generalized klt pairs. 
A conjecture called ``Generalised Nonvanishing Conjecture'' was formulated in \cite{lazic-peternell-gen-abund-I}, and we prove a variant of the conjecture. 

\begin{thm}[see also Theorem \ref{thm--non-vanishing-lc-divisor-2}]\label{thm--gen-non-vanishing-lp} Let $(X,B,\boldsymbol{\rm M})$ be a projective generalized lc pair such that $\boldsymbol{\rm M}$ is a finite $\mathbb{R}_{>0}$-linear combination of b-nef $\mathbb{Q}$-b-Cartier $\mathbb{Q}$-b-divisors. Let $A$ be an ample $\mathbb{R}$-divisor on $X$. Suppose that $K_{X}+B+A$ is pseudo-effective in the sense of Definition \ref{defn--positivity-divisor} (i.e., for every ample $\mathbb{R}$-divisor $H$ on $X$, the divisor $K_{X}+B+A+H$ is $\mathbb{R}$-linearly equivalent to an effective $\mathbb{R}$-divisor). 

Then, for every real number $\alpha \geq 0$, there exists an effective $\mathbb{R}$-divisor $D_{\alpha}$ on $X$ such that $K_{X}+B+A+\alpha \boldsymbol{\rm M}_{X} \sim_{\mathbb{R}}D_{\alpha}$ as (not necessarily $\mathbb{R}$-Cartier) $\mathbb{R}$-divisors.\end{thm}

When we consider $\mathbb{R}$-Cartier divisors, the pseudo-effectivity of Definition \ref{defn--positivity-divisor} coincides with the usual pseudo-effectivity. 
In Theorem \ref{thm--gen-non-vanishing-lp}, the pseudo-effectivity of $K_{X}+B+A$ holds, for example, when there is a resolution $f\colon \tilde{X} \to X$ and an effective $f$-exceptional $\mathbb{R}$-divisor $\tilde{E}$ on $\tilde{X}$ such that $f^{*}(K_{X}+B+\boldsymbol{\rm M}_{X}+A)-\boldsymbol{\rm M}_{\tilde{X}}+\tilde{E}$ is pseudo-effective in the usual sense. 


Although the above theorems are concerned with only the non-vanishing theorem, we can prove a stronger property (see Theorem \ref{thm--abund-main} for more general setup). 

\begin{thm}[see also Theorem \ref{thm--abund-main}]\label{thm--intro-abund-main}
Let $(X,B,\boldsymbol{\rm M})$ be a projective generalized lc pair such that $\boldsymbol{\rm M}$ is a finite $\mathbb{R}_{>0}$-linear combination of b-nef $\mathbb{Q}$-b-Cartier $\mathbb{Q}$-b-divisors. 
Let $A$ be an effective ample $\mathbb{R}$-divisor on $X$ such that $(X,B+A, \boldsymbol{\rm M})$ is a generalized lc pair and generalized lc centers of $(X,B+A, \boldsymbol{\rm M})$ are generalized lc centers of $(X,B, \boldsymbol{\rm M})$. 
Let $(Y,\Gamma,\boldsymbol{\rm M})$ be a $\mathbb{Q}$-factorial dlt model of $(X,B+A, \boldsymbol{\rm M})$. 

Then, for any sequence of steps of a $(K_{Y}+\Gamma+\boldsymbol{\rm M}_{Y})$-MMP
$$(Y,\Gamma,\boldsymbol{\rm M})=:(Y_{0},\Gamma_{0},\boldsymbol{\rm M}) \dashrightarrow \cdots \dashrightarrow (Y_{i},\Gamma_{i},\boldsymbol{\rm M}) \dashrightarrow \cdots,$$
the divisor $K_{Y_{i}}+\Gamma_{i}+\boldsymbol{\rm M}_{Y_{i}}$ is log abundant with respect to $(Y_{i},\Gamma_{i},\boldsymbol{\rm M})$ for every $i\geq0$.  
\end{thm}

It is important to consider the situation of Theorem \ref{thm--intro-abund-main} because the situation is deeply concerned with the existence of flips for $(K_{X}+B+\boldsymbol{\rm M}_{X})$-flipping contractions. 
Property of being log abundant for $\mathbb{R}$-divisors is much stronger than effectivity up to $\mathbb{R}$-linear equivalence (Subsection \ref{subsec2.2}), therefore Theorem \ref{thm--intro-abund-main} gives a partial answer to \cite[Question 3.5]{bhzariski} by Birkar and Hu. 
In the case of log pairs, the existence of log MMP with scaling of an ample divisor preserving property of being log abundant implies the existence of a log minimal model (\cite[Corollary 1.2]{has-finite}). We believe that Theorem \ref{thm--intro-abund-main} and the arguments in \cite[Section 5]{hashizumehu} (or \cite{has-finite}) with the aid of \cite{hanli} by Han--Li will give a proof of the existence of a minimal model in the setting of Theorem \ref{thm--intro-abund-main}. On the other hand, the abundance conjecture in the setting of Theorem \ref{thm--intro-abund-main} seems difficult because the Kodaira type vanishing theorem (\cite{fujino-nonvan}, \cite{fujino-fund}) and the finiteness of pluri-canonical representation (\cite{fujino-gongyo}, \cite{haconxu}) for generalized lc pairs are still widely open. 

Finally, we give a remark on Kodaira type vanishing theorem for generalized pairs. 
As a generalization of 
\cite[Theorem 5.6.4]{fujino-book}, for all projective generalized lc pairs $(X,B,\boldsymbol{\rm M})$ and $\mathbb{Q}$-Cartier Weil divisors $D$ on $X$ such that $D-(K_{X}+B+\boldsymbol{\rm M}_{X})$ is ample, we may expect that $H^{i}(X,\mathcal{O}_{X}(D))$ vanishes for every $i>0$. 
This problem will be an important step toward the cone and contraction theorem for generalized lc pairs. 
Since the non-vanishing theorem for generalized lc pairs holds in the situation of Theorem \ref{thm--abund-main-2} (\ref{thm--abund-main-2-(2)}), it is natural to expect that the Kodaira type vanishing theorem for generalized lc pairs also holds in the situation as in Theorem \ref{thm--abund-main-2} (\ref{thm--abund-main-2-(2)}) or  Theorem \ref{thm--gen-non-vanishing-lp}. 
Unfortunately, the following result shows that the expectation cannot be realized in general. 

\begin{thm}[see {Theorem \ref{thm--van-example}}]\label{thm--intro-van-example} 
There is a projective $\mathbb{Q}$-factorial generalized klt pair $(X,B,\boldsymbol{\rm M})$ and a Cartier divisor $D$ on $X$ satisfying the following property: There is a rational number $t>1$ such that $D-(K_{X}+B+t\boldsymbol{\rm M}_{X})$ is ample and $H^{1}(X,\mathcal{O}_{X}(D))\neq \{0\}$. 
\end{thm}


The contents of this paper are as follows: 
In Section \ref{sec2}, we collect definition and basic properties of generalized pairs, and we define property of being log abundant. 
In Section \ref{sec3}, we study the generalized abundance for generalized lc pairs. 
In Section \ref{sec4}, we prove 
theorems \ref{thm--abund-main-2}--\ref{thm--intro-van-example}. 
In Section \ref{sec5}, which is an appendix, we give a small remark on a non-$\mathbb{R}$-Cartier analogue of Theorem \ref{thm--abund-main-2} (\ref{thm--abund-main-2-(2)}).

\begin{ack}
The author was partially supported by JSPS KAKENHI Grant Number JP16J05875 and JP19J00046. 
Part of the work was done while the author was visiting University of Cambridge in October--November 2018 and January 2020. 
The author thanks staffs of the university, Professor Caucher Birkar, Doctor Roberto Svaldi, and Doctor Yanning Xu for their hospitality. 
He is grateful to Professor Caucher Birkar, Professor Yoshnori Gongyo, and Doctor Yanning Xu for discussions and giving him advice. 
He thanks Professors Osamu Fujino and Yoshinori Gongyo for comments. 
He thanks Doctor Sho Ejiri for answering questions. 
\end{ack}

\section{Preliminaries}\label{sec2}

\begin{note}
Throughout this paper, {\em $\mathbb{R}$-divisors} on varieties are not assumed to be $\mathbb{R}$-Cartier. 
In particular, big $\mathbb{R}$-divisors are not necessarily $\mathbb{R}$-Cartier. 
In this paper, we use non-negative integers as subscripts of varieties. 
For instance, varieties indexed by non-negative integers are denoted by $\tilde{X}_{i}$ with $i\in \mathbb{Z}_{\geq0}$. 
For any variety $W$, $\mathbb{R}$-divisor $D$ on $W$ and birational maps $W\dashrightarrow W'$ and $W\dashrightarrow W_{i}$, if there is no risk of confusion, the birational transform of $D$ on $W'$ (resp.~$W_{i}$) is denoted by $D'$ (resp.~$D_{i}$). 
\end{note}

\subsection{Generalized pairs}\label{subsec2.1}
In this subsection, we collect definition and properties of generalized pairs. 

\begin{b-divisor}
We define b-divisors. 

\begin{defn}[Birational model, divisors over varieties]\label{defn--b-divisor}
Let $X$ be a normal variety. 
We call a normal variety $Y$ with a projective birational morphism $Y\to X$ a {\em birational model}. 
When we want to mention the birational morphism $Y\to X$ explicitly, we call $Y\to X$ a {\em birational model}. 
A {\em prime divisor $P$ over $X$} is a prime divisor on a birational model $Y$. 
For any prime divisor $P$ over $X$, we denote the image of $P$ on $X$ by $c_{X}(P)$. 
For any prime divisor $P$ over $X$, let $v_{P}$ be the corresponding divisorial valuation. 
For any prime divisor $P$ on a birational model $Y\to X$ and $P'$ on a birational model $Y'\to X$, the divisorial valuations $v_{P}$ and $v_{P'}$ coincide if and only if the induced birational map $Y\dashrightarrow Y'$ is an isomorphism on a generic point of $P$ and the birational transform of $P$ on $Y'$ is equal to $P'$. 
\end{defn}

\begin{defn}[b-divisors]
Let $X$ be a normal variety. 
Then an {\em $\mathbb{R}$-b-divisor} $\boldsymbol{\rm D}$ on $X$ is a (possibly infinitely many) $\mathbb{R}$-linear combination of divisorial valuations $v_{P}$
$$\boldsymbol{\rm D}:=\sum_{\substack {P:{\rm \,prime\,divisor}\\{\rm over \,}X}} r_{P}v_{P} \quad (r_{P}\in \mathbb{R})$$
such that the set $\{P\;|\; \text{$r_{P}\neq0$ and $c_{X}(P)$ is a prime divisor on $Y$}\}$ is a finite set for all birational models $Y\to X$. 
When all $r_{P}$ are rational numbers, we call $\boldsymbol{\rm D}$ a {\em $\mathbb{Q}$-b-divisor}. 
For a given birational model $Y$, the {\em trace} of an $\mathbb{R}$-b-divisor $\boldsymbol{\rm D}=\sum_{P} r_{P}v_{P}$ on $Y$, which we denote $\boldsymbol{\rm D}_{Y}$, is defined by
$$\boldsymbol{\rm D}_{Y}:=\sum_{\substack {c_{Y}(P){\rm \,is \;a\;divisor}\\{\rm  on \,}Y}} r_{P}c_{Y}(P).$$
By definition, $\boldsymbol{\rm D}_{Y}$ is an $\mathbb{R}$-divisor on $Y$. 

Let $\boldsymbol{\rm D}$ be an $\mathbb{R}$-b-divisor on $X$. 
If there exists an $\mathbb{R}$-Cartier divisor $D$ on a birational model $Y$ such that 
 $$\boldsymbol{\rm D}=\sum_{\substack {P:{\rm \,prime\,divisor}\\{\rm over \,}Y}} {\rm ord}_{P}(D)\cdot v_{P},$$
 then, we call $\boldsymbol{\rm D}$ an {\em $\mathbb{R}$-b-Cartier $\mathbb{R}$-b-divisor}, we say that {\em $\boldsymbol{\rm D}$ descends to $Y$} or {\em $\boldsymbol{\rm D}$ is the closure of $D$}, and we write $\boldsymbol{\rm D}=\overline{D}$. 
In this situation, we have $\boldsymbol{\rm D}_{Y}=D$. 
When $\boldsymbol{\rm D}=\overline{D}$ and $D$ is $\mathbb{Q}$-Cartier, we call $\boldsymbol{\rm D}$ a {\em $\mathbb{Q}$-b-Cartier $\mathbb{Q}$-b-divisor}. 
When $\boldsymbol{\rm D}=\overline{D}$ and $X$ has a projective morphism $X\to Z$ to a variety $Z$ such that $D$ is nef over $Z$, we say that $\boldsymbol{\rm D}$ is a {\em b-nef$/Z$ $\mathbb{R}$-b-Cartier $\mathbb{R}$-b-divisor}.  
\end{defn}

\begin{rem}
Let $X$ be a normal variety and $\boldsymbol{\rm D}$ an $\mathbb{R}$-b-divisor on $X$. 
Let $\phi \colon X\dashrightarrow X'$ be a birational map to a normal variety $X'$. 
If there are proper birational morphisms $f\colon Y \to X$ and $f'\colon Y\to X'$ such that $f'=\phi\circ f$, then we may think of $\boldsymbol{\rm D}$ as an $\mathbb{R}$-b-divisor on $X'$. 
If $\boldsymbol{\rm D}$ is an $\mathbb{R}$-b-Cartier (resp.~b-nef $\mathbb{R}$-b-Cartier) $\mathbb{R}$-b-divisor on $X$, then $\boldsymbol{\rm D}$ which we see as an $\mathbb{R}$-b-divisor on $X'$ is an $\mathbb{R}$-b-Cartier (resp.~b-nef $\mathbb{R}$-b-Cartier) $\mathbb{R}$-b-divisor. 
Similar facts hold true in the framework of $\mathbb{Q}$-b-divisors. 
\end{rem}
\end{b-divisor}

\begin{sing}
A {\em pair} $(X,B)$ consists of a normal variety $X$ and effective $\mathbb{R}$-divisor $B$ on $X$ such that $K_{X}+B$ is $\mathbb{R}$-Cartier. 

Let $(X,B)$ be a pair and $P$ a prime divisor over $X$. 
Then $a(P,X,B)$ denotes the {\em log discrepancy} of $P$ with respect to $(X,B)$. 
A pair $(X,B)$ is called a {\em Kawamata log terminal} (klt, for short) {\em pair} if $a(P,X,B)>0$ for all prime divisors $P$ over $X$. 
A pair $(X,B)$ is called a {\em log canonical} (lc, for short) {\em pair} if $a(P,X,B)\geq0$ for all prime divisors $P$ over $X$. 
A pair $(X,B)$ is called a {\em divisorially log terminal} (dlt, for short) pair if all coefficients of $B$ belong to $[0,1]$ and there is a log resolution of $(X,{\rm Supp}B)$ whose exceptional divisors $E$ satisfy $a(E,X,B)>0$. 
When $(X,B)$ is an lc pair, an {\em lc center} of $(X,B)$ is the image on $X$ of a prime divisor $P$ over $X$ which satisfies $a(P,X,B)=0$. 
\end{sing}

\begin{g-pair}
We define generalized pairs and singularities of generalized pairs. 
In  \cite[Definition 1.4]{bz} the boundary part of generalized pairs are assumed to have coefficients in $[0,1]$, but, in Definition \ref{defn--gen-pair} below we only assume that the boundary parts are effective divisors.

\begin{defn}[Singularities of generalized pairs]\label{defn--gen-pair}
A {\em generalized pair} $(X,B,\boldsymbol{\rm M})/Z$ consists of 
a projective morphism $X\to Z$ from a normal variety to a variety, an effective $\mathbb{R}$-divisor $B$ on $X$, and a b-nef$/Z$ $\mathbb{R}$-b-Cartier $\mathbb{R}$-b-divisor $\boldsymbol{\rm M}$ on $X$ 
such that $K_{X}+B+\boldsymbol{\rm M}_{X}$ is $\mathbb{R}$-Cartier. 
When $\boldsymbol{\rm M}=0$, the generalized pair $(X,B,\boldsymbol{\rm M})/Z$ becomes a pair. 
When $Z$ is a point, we simply denote $(X,B,\boldsymbol{\rm M})$. 

Let $(X,B,\boldsymbol{\rm M})/Z$ be a generalized pair and $P$ a prime divisor over $X$. 
Let $f\colon \tilde{X} \to X$ be a projective birational morphism such that $\boldsymbol{\rm M}=\overline{\boldsymbol{\rm M}_{\tilde{X}}}$ and $P$ appears as a divisor on $\tilde{X}$. 
Then there is an $\mathbb{R}$-divisor $\tilde{B}$ on $\tilde{X}$ such that
$$K_{\tilde{X}}+\tilde{B}+\boldsymbol{\rm M}_{\tilde{X}}=f^{*}(K_{X}+B+\boldsymbol{\rm M}_{X}).$$
Then the {\em generalized log discrepancy $a(P,X,B+\boldsymbol{\rm M}_{X})$} of $P$ with respect to $(X,B,\boldsymbol{\rm M})/Z$ is defined by $a(P,X,B+\boldsymbol{\rm M}_{X})=1-{\rm coeff}_{P}(\tilde{B})$. 
The generalized log discrepancy does not depend on the choice of $f\colon \tilde{X} \to X$. 
When $\boldsymbol{\rm M}=0$, the generalized log discrepancies coincide with the log discrepancies of the pair $(X,B)$. 

A generalized pair $(X,B,\boldsymbol{\rm M})/Z$ is called a {\em generalized klt} (resp.~{\em generalized lc}) {\em pair} if $a(P,X,B+\boldsymbol{\rm M}_{X})>0$ (resp.~$a(P,X,B+\boldsymbol{\rm M}_{X})\geq 0$) for all prime divisors $P$ over $X$. 
A {\em generalized lc center} of $(X,B,\boldsymbol{\rm M})/Z$ is the image on $X$ of a prime divisor $P$ over $X$ satisfying $a(P,X,B+\boldsymbol{\rm M}_{X})=0$. 

A generalized pair $(X,B,\boldsymbol{\rm M})/Z$ is a {\em $\mathbb{Q}$-factorial generalized dlt pair} if $X$ is $\mathbb{Q}$-factorial, $(X,B)$ is a dlt pair, and $(X,B,(1+t)\boldsymbol{\rm M})/Z$ is generalized lc for some real number $t>0$. 
\end{defn}

For a given $(X,B,\boldsymbol{\rm M})/Z$, we call $(Y,\Gamma,\boldsymbol{\rm M})/Z$ in Theorem \ref{thm--dlt-model} below a {\em $\mathbb{Q}$-factorial dlt model} of $(X,B,\boldsymbol{\rm M})/Z$. 

\begin{thm}[Existence of a $\mathbb{Q}$-factorial generalized dlt model, cf.~{\cite[Lemma 4.5]{bz}}]\label{thm--dlt-model}
Let $(X,B,\boldsymbol{\rm M})/Z$ be a generalized lc pair such that $Z$ is quasi-projective. 
Then, there exists a projective birational morphism $f\colon Y\to X$ and a $\mathbb{Q}$-factorial generalized dlt pair $(Y,\Gamma,\boldsymbol{\rm M})/Z$ such that $K_{Y}+\Gamma+\boldsymbol{\rm M}_{Y}=f^{*}(K_{X}+B+\boldsymbol{\rm M}_{X})$ and $\Gamma$ is the sum of $f^{-1}_{*}B$ and the reduced $f$-exceptional divisor.  
\end{thm}

\begin{rem}\label{rem--gen-dlt}
There are two remarks on $\mathbb{Q}$-factorial generalized dlt pairs.

\begin{enumerate}[(1)]
\item \label{rem--gen-dlt-(1)}
The above definition of $\mathbb{Q}$-factorial generalized dlt pairs coincides with the definition of generalized dlt pairs in \cite[2.13 (2)]{birkar-compl} and \cite[Definition 2.2]{hanli} when $X$ is $\mathbb{Q}$-factorial. 
In this paper, we always assume generalized dlt pairs to be $\mathbb{Q}$-factorial. 
\item \label{rem--gen-dlt-(2)}
Let $(X,B,\boldsymbol{\rm M})/Z$ be a generalized lc pair such that $Z$ is quasi-projective. 
For every prime divisor $P$ over $X$ such that $a(P,X,B+\boldsymbol{\rm M}_{X})=0$, \cite[Lemma 4.5]{bz} shows that there is a $\mathbb{Q}$-factorial generalized dlt model $(Y,\Gamma,\boldsymbol{\rm M})\to(X,B,\boldsymbol{\rm M})$ on which $P$ appears as a divisor. 
In particular, for any lc center $S$ of $(X,B,\boldsymbol{\rm M})/Z$, we can construct a $\mathbb{Q}$-factorial generalized dlt model $(Y,\Gamma,\boldsymbol{\rm M})$ such that $X'\to X$ induces a surjective morphism $T \to S$ from a component $T$ of $\llcorner B'\lrcorner$. 
\end{enumerate}
\end{rem}
\end{g-pair}

\begin{lem}\label{lem--logsmooth}
Let $(X,B,\boldsymbol{\rm M})/Z$ be a generalized lc pair, and let $f\colon \tilde{X}\to X$ be a log resolution of $(X,{\rm Supp B})$ such that
 ${\boldsymbol{\rm M}}=\overline{\boldsymbol{\rm M}_{\tilde{X}}}$. 
We define $\tilde{B}\geq0$ and $\tilde{E}\geq0$ on $\tilde{X}$ by 
$$K_{\tilde{X}}+\tilde{B}+\boldsymbol{\rm M}_{\tilde{X}}=f^{*}(K_{X}+B+\boldsymbol{\rm M}_{X})+\tilde{E}$$
such that $\tilde{B}$ and $\tilde{E}$ have no common components. 
Let $g\colon Y\to \tilde{X}$ be a projective birational morphism from a smooth variety $Y$. 
We define $\Gamma\geq0$ and $F\geq0$ on $Y$ by $K_{Y}+\Gamma=g^{*}(K_{\tilde{X}}+\tilde{B})+F$ such that $\Gamma$ and $F$ have no common components. 

Then $\Gamma$ and $g^{*}\tilde{E}+F$ have no common components. 
We note that we have
$$K_{Y}+\Gamma+\boldsymbol{\rm M}_{Y}=(f \circ g)^{*}(K_{X}+B+\boldsymbol{\rm M}_{X})+g^{*}\tilde{E}+F$$
by construction. 
\end{lem}

\begin{proof}
By hypothesis, $\Gamma$ and $F$ have no common components. 
Hence, it is sufficient to show that $\Gamma$ and $g^{*}\tilde{E}$ have no common components. 
Since $\tilde{B}$ and $\tilde{E}$ have no common components, by log smoothness of $\bigl(\tilde{X},{\rm Supp}(\tilde{B}+\tilde{E})\bigr)$ and computations of discrepancies as in \cite[Lemma 2.29]{kollar-mori} and \cite[Lemma 2.45]{kollar-mori}, for any prime divisor $P$ over $\tilde{X}$ if the center on $\tilde{X}$ is contained in ${\rm Supp}\tilde{E}$ then $a(P,\tilde{X},\tilde{B})\geq 1$. 
Since the log discrepancies of every component of $\Gamma$ with respect to $(\tilde{X},\tilde{B})$ is less than $1$, we see that $\Gamma$ and $g^{*}\tilde{E}$ have no common components. 
Therefore, $\Gamma$ and $g^{*}\tilde{E}+F$ have no common components. 
\end{proof}




\begin{adj-g-pair}
We recall construction of a generalized pair by divisorial adjunction for generalized lc pairs (\cite[Definition 4.7]{bz}). 

Let $(X,B,\boldsymbol{\rm M})/Z$ be a generalized lc pair, and let $S$ be a component of $\llcorner B \lrcorner$ with the normalization $S^{\nu}\to S$. 
Let $f\colon \tilde{X}\to X$ be a log resolution of $(X,{\rm Supp}B)$ such that $\boldsymbol{\rm M}=\overline{\boldsymbol{\rm M}_{\tilde{X}}}$. 
We can write
$$K_{\tilde{X}}+\tilde{B}+\boldsymbol{\rm M}_{\tilde{X}}=f^{*}(K_{X}+B+\boldsymbol{\rm M}_{X})$$
with an $\mathbb{R}$-divisor $\tilde{B}$ on $\tilde{X}$. 
Put $\tilde{S}=f^{-1}_{*}S$ which is a component of $\llcorner \tilde{B} \lrcorner$. 
Pick $\tilde{M}\sim_{\mathbb{R}}\boldsymbol{\rm M}_{\tilde{X}}$ so that $\tilde{M}|_{\tilde{S}}$ is well-defined as an $\mathbb{R}$-divisor on $\tilde{S}$. 
We define a b-nef$/Z$ $\mathbb{R}$-b-divisor $\boldsymbol{\rm N}$ on $S^{\nu}$ by the closure of $\tilde{M}|_{\tilde{S}}$, and we set $B_{S^{\nu}}$ as the birational transform of $(\tilde{B}-\tilde{S})|_{\tilde{S}}$ by the induced birational morphism $f_{\tilde{S}}\colon \tilde{S}\to S^{\nu}$. 
Then it follows that $(S^{\nu},B_{S^{\nu}}, \boldsymbol{\rm N})/Z$ is a generalized pair and 
$$K_{\tilde{S}}+(\tilde{B}-\tilde{S})|_{\tilde{S}}+\boldsymbol{\rm N}_{\tilde{S}}=f_{\tilde{S}}^{*}(K_{S^{\nu}}+B_{S^{\nu}}+\boldsymbol{\rm N}_{S^{\nu}}).$$
Note that $\boldsymbol{\rm N}$ is determined up to $\mathbb{R}$-linear equivalence because $\boldsymbol{\rm N}$ depends on $\tilde{M}$. 
\end{adj-g-pair}

\begin{lem}\label{lem--adjunction}
Let $(X,B,\boldsymbol{\rm M})$ be a projective $\mathbb{Q}$-factorial generalized dlt pair, and let $f\colon \tilde{X}\to X$ be a log resolution of $(X,{\rm Supp}B)$ such that
 ${\boldsymbol{\rm M}}=\overline{\boldsymbol{\rm M}_{\tilde{X}}}$. 
Let $S$ be a component of $\llcorner B \lrcorner$. 
We define $\tilde{B}\geq0$ and $\tilde{E}\geq0$ on $\tilde{X}$ by $K_{\tilde{X}}+\tilde{S}+\tilde{B}+\boldsymbol{\rm M}_{\tilde{X}}=f^{*}(K_{X}+B+\boldsymbol{\rm M}_{X})+\tilde{E}$  such that $\tilde{S}=f_{*}^{-1}S$ and $\tilde{B}$, $\tilde{E}$, and $\tilde{S}$ have no common components one another. 
We define $\mathbb{R}$-divisors $B_{\tilde{S}}$, $B_{S}$, $M_{\tilde{S}}$, and $M_{S}$ by divisorial adjunction for generalized lc pairs. 
In other words, we define $B_{\tilde{S}}=\tilde{B}|_{\tilde{S}}$, $B_{S}=f|_{\tilde{S}*}B_{\tilde{S}}$, $M_{\tilde{S}}=\tilde{M}|_{\tilde{S}}$ for some $\tilde{M}\sim_{\mathbb{R}}\boldsymbol{\rm M}_{\tilde{X}}$, and $M_{S}=f|_{\tilde{S}*}M_{\tilde{S}}$. 

Then $(S,B_{S},\overline{M_{\tilde{S}}})$ and $(\tilde{S}, \tilde{B},\overline{M_{\tilde{S}}})$ are generalized lc pairs. 
Furthermore, we have 
$$a(P,S,B_{S}+M_{S})=a(P,\tilde{S}, B_{\tilde{S}}+M_{\tilde{S}})$$
for all prime divisors $P$ on $S$. 
\end{lem}

\begin{proof}
By definition of $B_{\tilde{S}}$ and $M_{\tilde{S}}$, we have 
$K_{\tilde{S}}+B_{\tilde{S}}+M_{\tilde{S}}=f|_{\tilde{S}}^{*}(K_{X}+B+\boldsymbol{\rm M}_{X})|_{S}+\tilde{E}|_{\tilde{S}}.$ 
The first assertion follows from \cite[Remark 4.8]{bz}. 
For the second assertion, the proof of \cite[Lemma 2.4]{hashizumehu} works with no changes. 
More precisely, by the same argument as in \cite[Proof of Lemma 2.4]{hashizumehu}, we see that $B_{\tilde{S}}$ and $\tilde{E}|_{\tilde{S}}$ have no common component and $\tilde{E}|_{\tilde{S}}$ is exceptional over $S$. 
Then the second assertion follows from the definition of log discrepancies for generalized pairs. 
\end{proof}

\begin{mmp-g-pair}
Let $(X,B,\boldsymbol{\rm M})/Z$ be a $\mathbb{Q}$-factorial generalized lc pair such that $Z$ is quasi-projective and $(X,0)$ is a klt pair. 
As in \cite[Section 4]{bz}, we can run a $(K_{X}+B+\boldsymbol{\rm M}_{X})$-MMP over $Z$ with scaling of an ample divisor. 
We often denote the sequence of the MMP by 
 $$(X,B,\boldsymbol{\rm M})=:(X_{0},B_{0},\boldsymbol{\rm M})\dashrightarrow \cdots \dashrightarrow (X_{i},B_{i},\boldsymbol{\rm M})\dashrightarrow \cdots.$$
Here, for each $i$ we think of $\boldsymbol{\rm M}$ as an $\mathbb{R}$-b-divisor on $X_{i}$. 

For fundamental results of MMP for generalized pairs, see \cite{hanli}. 
In this paper, we will freely use results in \cite{hanli}. 
\end{mmp-g-pair}

\subsection{Property of being log abundant}\label{subsec2.2}
In this subsection, we introduce the notion of being log abundant for generalized pairs. 

\begin{defn}[Invariant Iitaka dimension]\label{defn--inv-iitaka-dim}
Let $X$ be a normal projective variety, and let $D$ be an $\mathbb{R}$-Cartier  divisor on $X$. 
We define the {\em invariant  Iitaka dimension} of $D$, denoted by $\kappa_{\iota}(X,D)$, as follows (\cite[Definition 2.2.1]{choi}, see also \cite[Definition 2.5.5]{fujino-book}):  
If there is an $\mathbb{R}$-divisor $E\geq 0$ such that $D\sim_{\mathbb{R}}E$, set $\kappa_{\iota}(X,D)=\kappa(X,E)$. 
Here, the right hand side is the usual Iitaka dimension of $E$. 
Otherwise, we set $\kappa_{\iota}(X,D)=-\infty$. 
We can check that $\kappa_{\iota}(X,D)$ is well-defined, i.e., when there is $E\geq 0$ such that $D\sim_{\mathbb{R}}E$, the invariant Iitaka dimension $\kappa_{\iota}(X,D)$ does not depend on the choice of $E$. 
By definition, we have $\kappa_{\iota}(X,D)\geq0$ if and only if $D$ is $\mathbb{R}$-linearly equivalent to an effective $\mathbb{R}$-divisor. 
\end{defn}

\begin{defn}[Numerical dimension]\label{defn--num-dim}
Let $X$ be a normal projective variety, and let $D$ be an $\mathbb{R}$-Cartier divisor on $X$. 
We define the {\em numerical dimension} of $D$, denoted by $\kappa_{\sigma}(X,D)$, as follows (\cite[V, 2.5 Definition]{nakayama}): 
For any Cartier divisor $A$ on $X$, we set
\begin{equation*}
\sigma(D;A)={\rm max}\!\Set{\!k\in \mathbb{Z}_{\geq0} | \underset{m\to \infty}{\rm lim\,sup}\frac{{\rm dim}H^{0}(X,\mathcal{O}_{X}(\llcorner mD \lrcorner+A))}{m^{k}}>0\!}
\end{equation*}
if ${\rm dim}H^{0}(X,\mathcal{O}_{X}(\llcorner mD \lrcorner+A))>0$ for infinitely many $m\in \mathbb{Z}_{>0}$, and otherwise we set $\sigma(D;A):=-\infty$. 
Then, we define 
\begin{equation*}
\kappa_{\sigma}(X,D):={\rm max}\!\set{\sigma(D;A) | \text{$A$ is a Cartier divisor on $X$}\!}.
\end{equation*}
\end{defn}

See \cite[Remark 2.8]{hashizumehu} for basic properties of the invariant Iitaka dimension and the numerical dimension. 

\begin{defn}[Abundant divisors and log abundant divisors]\label{defn--abund}
Let $X$ be a normal projective variety, and let $D$ be an $\mathbb{R}$-Cartier divisor on $X$. 
We say that $D$ is abundant if the equality $\kappa_{\iota}(X,D)=\kappa_{\sigma}(X,D)$ holds. 

Let $X$ and $D$ be as above, and let $(X,B,\boldsymbol{\rm M})$ be a generalized lc pair. 
We say that $D$ is {\em log abundant with respect to $(X,B,\boldsymbol{\rm M})$} if $D$ is abundant and for any lc center $S$ of $(X,B,\boldsymbol{\rm M})$ with the normalization $S^{\nu}\to S$, the pullback $D|_{S^{\nu}}$ is abundant. 
\end{defn}

We can define the relative version of the invariant Iitaka dimension, the numerical dimension, and the property of being log abundant (\cite[Section 5]{has-finite}). 
But we will not use them in this paper.

\section{Generalized abundance for generalized lc pairs}\label{sec3}

\begin{lem}\label{lem--lift}
Let $(X,B,\boldsymbol{\rm M})$ be a projective generalized lc pair. 
Let $\phi\colon X \dashrightarrow X'$ be a birational contraction to a normal projective variety $X'$, and let $f\colon \tilde{X}\to X$ be a log resolution of $(X,{\rm Supp}B)$ such that $\boldsymbol{\rm M}=\overline{\boldsymbol{\rm M}_{\tilde{X}}}$ and the induced birational map $\tilde{X}\to X'$ is a morphism. 
Let $(\tilde{X},\tilde{B},\boldsymbol{\rm M})$ be a generalized lc pair. 
Suppose that 
\begin{itemize}
\item 
$K_{X'}+\phi_{*}B+\boldsymbol{\rm M}_{X'}$ is $\mathbb{R}$-Cartier, 
\item 
$a(P,X,B+\boldsymbol{\rm M}_{X})\leq a(P,X',\phi_{*}B+\boldsymbol{\rm M}_{X'})$ for all prime divisors $P$ on $\tilde{X}$, and
\item
the relation $K_{\tilde{X}}+\tilde{B}+\boldsymbol{\rm M}_{\tilde{X}}=f^{*}(K_{X}+B+\boldsymbol{\rm M}_{X})+\tilde{E}$ holds for an effective $f$-exceptional $\mathbb{R}$-divisor $\tilde{E}$. 
\end{itemize}
Then, by running a $(K_{\tilde{X}}+\tilde{B}+\boldsymbol{\rm M}_{\tilde{X}})$-MMP over $X'$ with scaling of an ample divisor we get a projective birational morphism $f'\colon \tilde{X}' \to X'$ such that 
$$K_{\tilde{X}'}+\tilde{B}'+\boldsymbol{\rm M}_{\tilde{X}'}=f'^{*}(K_{X'}+\phi_{*}B+\boldsymbol{\rm M}_{X'}),$$
where $\tilde{B}'$ is the birational transforms of $\tilde{B}$ on $\tilde{X}'$. 
\end{lem}

\begin{proof}
We denote $\tilde{X}\to X'$ by $g$. 
By the hypothesis of relation on discrepancies of the generalized pairs, there is a $g$-exceptional $\mathbb{R}$-divisor $\tilde{F}\geq 0$ on $\tilde{X}$ such that 
$$K_{\tilde{X}}+\tilde{B}+\boldsymbol{\rm M}_{\tilde{X}}=f^{*}(K_{X}+B+\boldsymbol{\rm M}_{X})+\tilde{E}=g^{*}(K_{X'}+\phi_{*}B+\boldsymbol{\rm M}_{X'})+\tilde{E}+\tilde{F}.$$
By construction, $\tilde{E}+\tilde{F}$ is effective and $g$-exceptional. 
By \cite[Proposition 3.8]{hanli} and running a $(K_{\tilde{X}}+\tilde{B}+\boldsymbol{\rm M}_{\tilde{X}})$-MMP over $X'$ with scaling of an ample divisor, we get a birational contraction $\tilde{X}\dashrightarrow \tilde{X}'$ over $X'$ which contracts $\tilde{E}+\tilde{F}$. 
Then it is clear that the morphism $\tilde{X}' \to X'$ is the desired one. 
\end{proof}

\begin{prop}\label{prop--relmmp}
Let $(X,B,\boldsymbol{\rm M})/Z$ be a generalized lc pair such that $Z$ is quasi-projective. 
Let $f\colon \tilde{X}\to X$ be a log resolution of $(X,{\rm Supp}B)$ such that
 ${\boldsymbol{\rm M}}=\overline{\boldsymbol{\rm M}_{\tilde{X}}}$, 
and let $\tilde{B}$ be an $\mathbb{R}$-divisor on $\tilde{X}$ such that $(\tilde{X},\tilde{B})$ is a log smooth lc pair and $f_{*}\tilde{B}=B$. 
Suppose that 
\begin{itemize}
\item
$K_{\tilde{X}}+\tilde{B}+\boldsymbol{\rm M}_{\tilde{X}}$ is pseudo-effective over $Z$,
\item
$K_{X}+B+\boldsymbol{\rm M}_{X}\sim_{\mathbb{R},Z}0$, and
\item
there is a klt pair $(X,\Delta)$ such that $\Delta$ is big over $Z$ and $K_{X}+\Delta\sim_{\mathbb{R},Z}0$. 
\end{itemize}
Then there is a birational contraction $\tilde{X}\dashrightarrow \tilde{X}'$ over $Z$ such that if $\tilde{B}'$ is the birational transforms of $\tilde{B}$ on $\tilde{X}'$, then
\begin{itemize}
\item
$K_{\tilde{X}'}+\tilde{B}'+\boldsymbol{\rm M}_{\tilde{X}'}$ is semi-ample over $Z$, and
\item
$a(P,\tilde{X},\tilde{B}+\boldsymbol{\rm M}_{\tilde{X}})\leq a(P,\tilde{X}',\tilde{B}'+\boldsymbol{\rm M}_{\tilde{X}'})$ for all prime divisors $P$ over $\tilde{X}$. 
\end{itemize}
\end{prop}

\begin{proof}
We prove it with several steps. 
\begin{step2}\label{step1relmmp}
By hypothesis, we can write
$$K_{\tilde{X}}+\tilde{B}+\boldsymbol{\rm M}_{\tilde{X}}=f^{*}(K_{X}+B+\boldsymbol{\rm M}_{X})+\tilde{E}-\tilde{F}\sim_{\mathbb{R},Z}\tilde{E}-\tilde{F},$$
where $\tilde{E}\geq0$ and $\tilde{F}\geq0$ are $f$-exceptional $\mathbb{R}$-divisors having no common components. 
Let $\tilde{G}$ be the sum of all $f$-exceptional prime divisors which are neither components of $\tilde{B}$ nor $\tilde{F}$. 
By definition, the divisor $\tilde{B}+\tilde{G}+\tilde{F}$ contains all $f$-exceptional prime divisors. 
Since $(X,B,\boldsymbol{\rm M})$ is generalized lc, we have 
${\rm coeff}_{P}(\tilde{B}+\tilde{F}-\tilde{E})\leq1$ for all prime divisors $P$ on $\tilde{X}$. 
If $P$ is a component of $\tilde{F}$, then 
$${\rm coeff}_{P}(\tilde{B}+\tilde{F})={\rm coeff}_{P}(\tilde{B}+\tilde{F}-\tilde{E})\leq1$$ because $\tilde{E}$ and $\tilde{F}$ have no common components. 
If $P$ is not a component of $\tilde{F}$, then 
$${\rm coeff}_{P}(\tilde{B}+\tilde{F})={\rm coeff}_{P}(\tilde{B})\leq 1$$
by definition of $\tilde{B}$. 
Thus, all prime divisors $P$ on $\tilde{X}$ satisfy ${\rm coeff}_{P}(\tilde{B}+\tilde{F})\leq1$. 
Since $\bigl(\tilde{X},{\rm Supp}\tilde{B}\cup{\rm Ex}(f)\bigr)$ is log smooth, there is $t_{0}\in(0,1)$ such that $\bigl(\tilde{X},\tilde{B}+t_{0}(\tilde{G}+\tilde{F})\bigr)$ is a log smooth lc pair. 
Then $\bigl(\tilde{X},\tilde{B}+t_{0}(\tilde{G}+\tilde{F}), {\boldsymbol{\rm M}})$ is a generalized lc pair with the nef part ${\boldsymbol{\rm M}}=\overline{\boldsymbol{\rm M}_{\tilde{X}}}$, and we have 
$$K_{\tilde{X}}+\tilde{B}+t_{0}(\tilde{G}+\tilde{F})+\boldsymbol{\rm M}_{\tilde{X}}\sim_{\mathbb{R},Z}+\tilde{E}+t_{0}\tilde{G}-(1-t_{0})\tilde{F}.$$
Because $\tilde{E}$ and $\tilde{F}$ have no common components and the same property holds for $\tilde{G}$ and $\tilde{F}$ by definition of $\tilde{G}$, the divisors $\tilde{E}+t_{0}\tilde{G}$ and $\tilde{F}$ have no common components. 
\end{step2}

\begin{step2}\label{step2relmmp}
We run a $(K_{\tilde{X}}+\tilde{B}+\boldsymbol{\rm M}_{\tilde{X}})$-MMP over $X$ with scaling of an ample divisor. 
Since $\tilde{E}\geq 0$ is $f$-exceptional, by \cite[Proposition 3.8]{hanli} and regarding $(\tilde{X},\tilde{B},{\boldsymbol{\rm M}})$ as a generalized lc pair over $X$, after finitely many steps we get a morphism $f_{0}\colon \tilde{X}_{0} \to X$ such that $\tilde{E}$ is contracted by the map $\tilde{X}\dashrightarrow \tilde{X}_{0}$ of the MMP. 
Let $\tilde{B}_{0}$, $\tilde{F}_{0}$ and $\tilde{G}_{0}$ be the birational transforms of $\tilde{B}$, $\tilde{F}$ and $\tilde{G}$ on $\tilde{X}_{0}$, respectively. 
Then the relation
$K_{\tilde{X}_{0}}+\tilde{B}_{0}+\boldsymbol{\rm M}_{\tilde{X}_{0}}\sim_{\mathbb{R},Z}-\tilde{F}_{0}$
holds and the inequality 
\begin{equation*}\tag{$*$}\label{proof--prop--relmmp-(*)}
a\bigl(P,\tilde{X}, \tilde{B}+\boldsymbol{\rm M}_{\tilde{X}}\bigr)\leq a\bigl(P,\tilde{X}_{0}, \tilde{B}_{0}+\boldsymbol{\rm M}_{\tilde{X}_{0}}\bigr)
\end{equation*} 
holds for all prime divisors $P$ over $\tilde{X}$. 
By the first condition of Proposition \ref{prop--relmmp}, we see that $K_{\tilde{X}_{0}}+\tilde{B}_{0}+\tilde{M}_{0}$ is pseudo-effective over $Z$, so $\tilde{F}_{0}$ is vertical over $Z$. 
In particular, there is $\tilde{H}_{0}\geq0$ on $\tilde{X}_{0}$ such that $K_{\tilde{X}_{0}}+\tilde{B}_{0}+\boldsymbol{\rm M}_{\tilde{X}_{0}}\sim_{\mathbb{R},Z}\tilde{H}_{0}$. 
For any $t\in(0.t_{0}]$, we have  
\begin{equation*}\tag{$\spadesuit$}\label{proof--prop--relmmp-(spadesuit)}
K_{\tilde{X}_{0}}+\tilde{B}_{0}+t(\tilde{G}_{0}+\tilde{F}_{0})+\boldsymbol{\rm M}_{\tilde{X}_{0}}\sim_{\mathbb{R},Z}t\tilde{G}_{0}-(1-t)\tilde{F}_{0}\sim_{\mathbb{R},Z}t(\tilde{G}_{0}+\tilde{F}_{0})+\tilde{H}_{0}. 
\end{equation*} 

Replacing $t_{0}>0$ in Step \ref{step1relmmp} with a small one, we may assume that the birational map $\tilde{X}\dashrightarrow \tilde{X}_{0}$ is a sequence of steps of a $(K_{\tilde{X}}+\tilde{B}+t_{0}(\tilde{G}+\tilde{F})+\boldsymbol{\rm M}_{\tilde{X}})$-MMP. 
Then the generalized pair $\bigl(\tilde{X}_{0}, \tilde{B}_{0}+t_{0}(\tilde{G}_{0}+\tilde{F}_{0}),\boldsymbol{\rm M}\bigr)$ is a $\mathbb{Q}$-factorial generalized dlt pair. 
\end{step2}

\begin{step2}\label{step3relmmp}
Pick a strictly decreasing sequence of positive real numbers $\{t_{i}\}_{i\geq1}$ such that $0<t_{i}<t_{0}$ for any $i$ and ${\rm lim}_{i\to \infty}t_{i}=0$. 
In this step, we prove that for each $i\geq1$ there exists a sequence of steps of a $\bigl(K_{\tilde{X}_{0}}+\tilde{B}_{0}+t_{i}(\tilde{G}_{0}+\tilde{F}_{0})+\boldsymbol{\rm M}_{\tilde{X}_{0}}\bigr)$-MMP over $Z$
$$\bigl(\tilde{X}_{0},\tilde{B}_{0}+t_{i}(\tilde{G}_{0}+\tilde{F}_{0}),\boldsymbol{\rm M}\bigr)\dashrightarrow \bigl(\tilde{X}_{i},\tilde{B}_{i}+t_{i}(\tilde{G}_{i}+\tilde{F}_{i}),\boldsymbol{\rm M}\bigr)$$
such that $\tilde{G}_{0}$ is contracted by the map $\tilde{X}_{0}\dashrightarrow \tilde{X}_{i}$ (i.e., $\tilde{G}_{i}=0$) and $K_{\tilde{X}_{i}}+\tilde{B}_{i}+t_{i}\tilde{F}_{i}+\boldsymbol{\rm M}_{\tilde{X}_{i}}$ is semi-ample over $Z$. 
From now to the end of this step, we fix $i$.  

We run a $\bigl(K_{\tilde{X}_{0}}+\tilde{B}_{0}+t_{i}(\tilde{G}_{0}+\tilde{F}_{0})+\boldsymbol{\rm M}_{\tilde{X}_{0}}\bigr)$-MMP over $X$ with scaling of an ample divisor. 
By (\ref{proof--prop--relmmp-(spadesuit)}) in Step \ref{step2relmmp}, this MMP is an MMP for $t_{i}\tilde{G}_{0}-(1-t_{i})\tilde{F}_{0}$ over $X$. 
Regarding $\bigl(\tilde{X}_{0},\tilde{B}_{0}+t_{i}(\tilde{G}_{0}+\tilde{F}_{0}),\boldsymbol{\rm M}\bigr)$ as a $\mathbb{Q}$-factorial generalized dlt pair over $X$ and applying \cite[Proposition 3.8]{hanli}, after finitely many steps we get a morphism 
$f'_{0}\colon \tilde{X}'_{0} \to X$ such that $\tilde{G}_{0}$ is contracted by the birational map $\tilde{X}_{0}\dashrightarrow \tilde{X}'_{0}$ of the MMP. 
In this way, we obtain a sequence of steps of a $\bigl(K_{\tilde{X}_{0}}+\tilde{B}_{0}+t_{i}(\tilde{G}_{0}+\tilde{F}_{0})+\boldsymbol{\rm M}_{\tilde{X}_{0}}\bigr)$-MMP over $Z$
$$\bigl(\tilde{X}_{0},\tilde{B}_{0}+t_{i}(\tilde{G}_{0}+\tilde{F}_{0}),\boldsymbol{\rm M}\bigr)\dashrightarrow \bigl(\tilde{X}'_{0}, \tilde{B}'_{0}+t_{i}\tilde{F}'_{0},\boldsymbol{\rm M}\bigr).$$
Here, the divisors $\tilde{B}'_{0}$ and $\tilde{F}'_{0}$ are the birational transforms of $\tilde{B}_{0}$ and $\tilde{F}_{0}$, respectively. 

Recall that $\tilde{B}+\tilde{G}+\tilde{F}$ contains all $f$-exceptional prime divisors. 
Therefore, $\tilde{B}'_{0}+t_{i}\tilde{F}'_{0}$ contains all $f'_{0}$-exceptional prime divisors in its support. 
Let $\tilde{\Delta}'_{0}$ be an $\mathbb{R}$-divisor on $\tilde{X}'_{0}$ defined by an equation $K_{\tilde{X}'_{0}}+\tilde{\Delta}'_{0}=f'^{*}_{0}(K_{X}+\Delta)$, where $\Delta$ is as in the third condition of Proposition \ref{prop--relmmp}. 
Since $\Delta$ is big over $Z$ and $\tilde{B}'_{0}+t_{i}\tilde{F}'_{0}$ contains all $f'_{0}$-exceptional prime divisors in its support, there is $u>0$ such that 
$$\tilde{\Psi}'_{0}:=\frac{1}{1+u}(\tilde{B}'_{0}+t_{i}\tilde{F}'_{0})+\frac{u}{1+u}\tilde{\Delta}'_{0}$$ is effective and big over $Z$. 
Since $(X,\Delta)$ is klt and $\bigl(\tilde{X}'_{0}, \tilde{B}'_{0}+t_{i}\tilde{F}'_{0},\boldsymbol{\rm M}\bigr)$ is generalized lc, the generalized pair $\bigl(\tilde{X}'_{0}, \tilde{\Psi}'_{0},\frac{1}{1+u}\boldsymbol{\rm M}\bigr)$ is generalized klt. 
Since $\tilde{\Psi}'_{0}$ is big over $Z$ and $\bigl(\tilde{X}'_{0}, \tilde{\Psi}'_{0},\frac{1}{1+u}\boldsymbol{\rm M}\bigr)$ is generalized klt, 
there exists an $\mathbb{R}$-divisor $\tilde{\Gamma}'_{0}\geq0$ on $\tilde{X}'_{0}$ which is big over $Z$ such that $\tilde{\Gamma}'_{0}\sim_{\mathbb{R},Z}\tilde{\Psi}'_{0}+\tfrac{1}{1+u}\boldsymbol{\rm M}_{\tilde{X}'_{0}}$ and the pair $\bigl(\tilde{X}'_{0}, \tilde{\Gamma}'_{0})$ is klt. 
By \cite{bchm}, there is a sequence of steps of a $(K_{\tilde{X}'_{0}}+\tilde{\Gamma}'_{0})$-MMP over $Z$ terminating with a good minimal model. 
By construction, we have
\begin{equation*}
\begin{split}
K_{\tilde{X}'_{0}}+\tilde{\Gamma}'_{0}\sim_{\mathbb{R},Z}&K_{\tilde{X}'_{0}}+\tilde{\Psi}'_{0}+\frac{1}{1+u}\boldsymbol{\rm M}_{\tilde{X}'_{0}}\\
=&K_{\tilde{X}'_{0}}+\frac{1}{1+u}(\tilde{B}'_{0}+t_{i}\tilde{F}'_{0})+\frac{u}{1+u}\tilde{\Delta}'_{0}+\frac{1}{1+u}\boldsymbol{\rm M}_{\tilde{X}'_{0}}\\
=&\frac{1}{1+u}\bigl(K_{\tilde{X}'_{0}}+\tilde{B}'_{0}+t_{i}\tilde{F}'_{0}+\boldsymbol{\rm M}_{\tilde{X}'_{0}}\bigr)+\frac{u}{1+u}\bigl(K_{\tilde{X}'_{0}}+\tilde{\Delta}'_{0}\bigr)\\
\sim_{\mathbb{R},Z}&\frac{1}{1+u}\bigl(K_{\tilde{X}'_{0}}+\tilde{B}'_{0}+t_{i}\tilde{F}'_{0}+\boldsymbol{\rm M}_{\tilde{X}'_{0}}\bigr). 
\end{split}
\end{equation*}
The final relation follows from $K_{\tilde{X}'_{0}}+\tilde{\Delta}'_{0}=f'^{*}_{0}(K_{X}+\Delta)\sim_{\mathbb{R},Z}0$. 
From this discussion, there is a sequence of steps of a $\bigl(K_{\tilde{X}'_{0}}+\tilde{B}'_{0}+t_{i}\tilde{F}'_{0}+\boldsymbol{\rm M}_{\tilde{X}'_{0}}\bigr)$-MMP over $Z$
$$\bigl(\tilde{X}'_{0}, \tilde{B}'_{0}+t_{i}\tilde{F}'_{0},\boldsymbol{\rm M}\bigr)\dashrightarrow \bigl(\tilde{X}_{i}, \tilde{B}_{i}+t_{i}\tilde{F}_{i}, \boldsymbol{\rm M}\bigr)$$
such that $K_{\tilde{X}_{i}}+\tilde{B}_{i}+t_{i}\tilde{F}_{i}+\boldsymbol{\rm M}_{\tilde{X}_{i}}$ is semi-ample over $Z$. 
Then the composition 
$$\bigl(\tilde{X}_{0},\tilde{B}_{0}+t_{i}(\tilde{G}_{0}+\tilde{F}_{0}),\boldsymbol{\rm M}\bigr)\dashrightarrow \bigl(\tilde{X}'_{0}, \tilde{B}'_{0}+t_{i}\tilde{F}'_{0},\boldsymbol{\rm M}\bigr)\dashrightarrow \bigl(\tilde{X}_{i}, \tilde{B}_{i}+t_{i}\tilde{F}_{i},\boldsymbol{\rm M}\bigr)$$
is the desired $\bigl(K_{\tilde{X}_{0}}+\tilde{B}_{0}+t_{i}(\tilde{G}_{0}+\tilde{F}_{0})+\boldsymbol{\rm M}_{\tilde{X}_{0}}\bigr)$-MMP over $Z$. 
We finish this step. 
\end{step2}

\begin{step2}\label{step4relmmp}
With this step we complete the proof. 

For any $i\geq 1$, we denote a sequence of steps of a $\bigl(K_{\tilde{X}_{0}}+\tilde{B}_{0}+t_{i}(\tilde{G}_{0}+\tilde{F}_{0})+\boldsymbol{\rm M}_{\tilde{X}_{0}}\bigr)$-MMP over $Z$ constructed in Step \ref{step3relmmp} by
$$\bigl(\tilde{X}_{0},\tilde{B}_{0}+t_{i}(\tilde{G}_{0}+\tilde{F}_{0}),\boldsymbol{\rm M}\bigr)\dashrightarrow \bigl(\tilde{X}_{i},\tilde{B}_{i}+t_{i}\tilde{F}_{i}, \boldsymbol{\rm M}).$$
We recall that $\tilde{G}_{0}$ is contracted by the map $\tilde{X}_{0}\dashrightarrow \tilde{X}_{i}$ and $K_{\tilde{X}_{i}}+\tilde{B}_{i}+t_{i}\tilde{F}_{i}+\boldsymbol{\rm M}_{\tilde{X}_{i}}$ is semi-ample over $Z$. 
By (\ref{proof--prop--relmmp-(spadesuit)}) in Step \ref{step2relmmp}, we see that prime divisors contracted by the map $\tilde{X}_{0}\dashrightarrow \tilde{X}_{i}$ are components of $\tilde{G}_{0}+\tilde{F}_{0}+\tilde{H}_{0}$, which are independent of $t_{i}$. 
Therefore, by replacing $\{t_{i}\}_{i\geq 1}$ with a subsequence, we may assume that all maps $\tilde{X}_{0}\dashrightarrow \tilde{X}_{i}$ contract the same divisors, so all $X_{i}$ are isomorphic in codimension one. 

We prove that the birational contraction $\tilde{X}\dashrightarrow \tilde{X}_{1}$ satisfies the two conditions of Proposition \ref{prop--relmmp}.   
We first check that $K_{\tilde{X}_{1}}+\tilde{B}_{1}+\boldsymbol{\rm M}_{\tilde{X}_{1}}$ is semi-ample over $Z$. 
Because $\tilde{G}$ is contracted by the map $\tilde{X}\dashrightarrow \tilde{X}_{1}$, with (\ref{proof--prop--relmmp-(spadesuit)}) in Step \ref{step2relmmp}, we have
\begin{equation*}
\begin{split}
K_{\tilde{X}_{1}}+\tilde{B}_{1}+t_{1}\tilde{F}_{1}+\boldsymbol{\rm M}_{\tilde{X}_{1}}\sim_{\mathbb{R},Z}-(1-t_{1})\tilde{F}_{1}.
\end{split}
\end{equation*}
From this relation and the fact that $K_{\tilde{X}_{1}}+\tilde{B}_{1}+t_{1}\tilde{F}_{1}+\boldsymbol{\rm M}_{\tilde{X}_{1}}$ is semi-ample over $Z$, we see that 
\begin{equation*}
\begin{split}
K_{\tilde{X}_{1}}+\tilde{B}_{1}+\boldsymbol{\rm M}_{\tilde{X}_{1}}\sim_{\mathbb{R},Z}-\tilde{F}_{1}\sim_{\mathbb{R},Z}\frac{1}{1-t_{1}}\bigl(K_{\tilde{X}_{1}}+\tilde{B}_{1}+t_{1}\tilde{F}_{1}+\boldsymbol{\rm M}_{\tilde{X}_{1}}\bigr)
\end{split}
\end{equation*}
is semi-ample over $Z$. 
Therefore, the first condition of Proposition \ref{prop--relmmp} holds true. 

Next, we prove $a(P,\tilde{X},\tilde{B}+\boldsymbol{\rm M}_{\tilde{X}})\leq a(P,\tilde{X}_{1},\tilde{B}_{1}+\boldsymbol{\rm M}_{\tilde{X}_{1}})$ for all prime divisors $P$ over $\tilde{X}$. 
Since $K_{\tilde{X}_{1}}+\tilde{B}_{1}+\boldsymbol{\rm M}_{\tilde{X}_{1}}$ is semi-ample over $Z$, we see that
\begin{equation*}
\begin{split}
K_{\tilde{X}_{1}}+\tilde{B}_{1}+t_{i}\tilde{F}_{1}+\boldsymbol{\rm M}_{\tilde{X}_{1}}\sim_{\mathbb{R},Z}-(1-t_{i})\tilde{F}_{1}\sim_{\mathbb{R},Z}(1-t_{i})\bigl(K_{\tilde{X}_{1}}+\tilde{B}_{1}+\boldsymbol{\rm M}_{\tilde{X}_{1}}\bigr)
\end{split}
\end{equation*}
is semi-ample over $Z$ for any $i$. 
By construction, the induced birational map $\tilde{X}_{1}\dashrightarrow \tilde{X}_{i}$ are isomorphic in codimension one and $\tilde{B}_{i}$, $\tilde{F}_{i}$, and $\boldsymbol{\rm M}_{\tilde{X}_{i}}$ are the birational transforms of $\tilde{B}_{1}$, $\tilde{F}_{1}$, and $\boldsymbol{\rm M}_{\tilde{X}_{1}}$ on $\tilde{X}_{i}$, respectively. 
Since $K_{\tilde{X}_{i}}+\tilde{B}_{i}+t_{i}\tilde{F}_{i}+\boldsymbol{\rm M}_{\tilde{X}_{i}}$ is semi-ample over $Z$, by taking a common resolution of $\tilde{X}_{1}\dashrightarrow \tilde{X}_{i}$ and the negativity lemma, we obtain
$$a(P,\tilde{X}_{i},\tilde{B}_{i}+t_{i}\tilde{F}_{i}+\boldsymbol{\rm M}_{\tilde{X}_{i}})= a(P,\tilde{X}_{1},\tilde{B}_{1}+t_{i}\tilde{F}_{1}+\boldsymbol{\rm M}_{\tilde{X}_{1}})$$
for any $i>1$. 
Combining this with construction of the map $\tilde{X}_{0}\dashrightarrow \tilde{X}_{i}$ in Step \ref{step3relmmp}, we obtain $a(P,\tilde{X}_{0},\tilde{B}_{0}+t_{i}(\tilde{G}_{0}+\tilde{F}_{0})+\boldsymbol{\rm M}_{\tilde{X}_{0}})\leq a(P,\tilde{X}_{1},\tilde{B}_{1}+t_{i}\tilde{F}_{1}+\boldsymbol{\rm M}_{\tilde{X}_{1}})$ for any $i$. 
By taking the limit $i\to \infty$, we have 
$$a(P,\tilde{X}_{0},\tilde{B}_{0}+\boldsymbol{\rm M}_{\tilde{X}_{0}})\leq a(P,\tilde{X}_{1},\tilde{B}_{1}+\boldsymbol{\rm M}_{\tilde{X}_{1}}).$$
We combine this inequality and (\ref{proof--prop--relmmp-(*)}) in Step \ref{step2relmmp}, then we obtain
$$a\bigl(P,\tilde{X}, \tilde{B}+\boldsymbol{\rm M}_{\tilde{X}}\bigr)\leq a(P,\tilde{X}_{0},\tilde{B}_{0}+\boldsymbol{\rm M}_{\tilde{X}_{0}})\leq a(P,\tilde{X}_{1},\tilde{B}_{1}+\boldsymbol{\rm M}_{\tilde{X}_{1}}).$$
In this way, we see that $\tilde{X}\dashrightarrow \tilde{X}_{1}$ is the desired birational contraction.  
\end{step2}
We complete the proof. 
\end{proof}

\begin{rem}
When $B$ and ${\boldsymbol{\rm M}}$ have rational coefficients, Proposition \ref{prop--relmmp} follows from \cite[Lemma 4.10]{birkar-connectedness} and arguments in \cite[Proof of Lemma 4.14]{birkar-connectedness}. 
These arguments in \cite[Proofs of Lemma 4.10 and Lemma 4.14]{birkar-connectedness} work in our situation. 
We note that the arguments are different from our proof. 
\end{rem}

\begin{prop}\label{prop--genabklt}
Let $\pi \colon X\to Z$ be a morphism of normal projective varieties, and let $(X,B, \boldsymbol{\rm M})$ be a projective generalized klt pair. 
Suppose that
\begin{itemize}
\item
$\boldsymbol{\rm M}$ is a finite $\mathbb{R}_{>0}$-linear combination of b-nef $\mathbb{Q}$-b-Cartier $\mathbb{Q}$-b-divisors, and 
\item
there exists an open subset $U\subset Z$ such that $(K_{X}+B+\boldsymbol{\rm M}_{X})|_{\pi^{-1}(U)}\sim_{\mathbb{R},U}0$. 
\end{itemize} 
Let $A_{Z}$ be an ample $\mathbb{R}$-divisor on $Z$.  
 
Then, $K_{X}+B+\boldsymbol{\rm M}_{X}+\pi^{*}A_{Z}$ is abundant. 
Furthermore, when $K_{X}+B+\boldsymbol{\rm M}_{X}+\pi^{*}A_{Z}$ is pseudo-effective, the divisor birationally has the Nakayama--Zariski decomposition with semi-ample positive part. 
\end{prop}

\begin{proof}
The argument is very similar to \cite[Proof of Lemma 3.1]{has-mmp}. 
\begin{step3}\label{step1genabklt}
By taking a Stein factorization of $\pi$, we may assume that $\pi$ is a contraction. 
Let $\overline{f}\colon \overline{X}\to X$ be a log resolution of $(X,{\rm Supp}B)$ such that $\boldsymbol{\rm M}=\overline{\boldsymbol{\rm M}_{\overline{X}}}$. 
By construction, we can find $\mathbb{R}$-divisors $\overline{B}$ and $\overline{E}$ on $\overline{X}$ such that $(\overline{X},\overline{B})$ is a log smooth klt pair, $\overline{E}$ is $\overline{f}$-exceptional, and
$$K_{\overline{X}}+\overline{B}+\boldsymbol{\rm M}_{\overline{X}}=\overline{f}^{*}(K_{X}+B+\boldsymbol{\rm M}_{X})+\overline{E}.$$ 
Note that $\overline{B}$ and $\overline{E}$ may have common components. 
We apply the weak semistable reduction (\cite[Proof of Theorem 2.1]{ak}, \cite[Proposition 4.4]{ak} and \cite[Remark 4.5]{ak}) to the morphism $(\overline{X},\overline{B})\to Z$. 
Then we have the following diagram 
 $$
\xymatrix{
\overline{X} \ar[d]_{\pi\circ \overline{f}}&\tilde{X} \ar[l]_{\tilde{f}}\ar[d]^{\tilde{\pi}}\\
Z&\tilde{Z}\ar[l]^{g}
}
$$
with normal projective varieties $\tilde{X}$ and a smooth projective variety $\tilde{Z}$ such that 
\begin{itemize}
\item
$\tilde{f}\colon \tilde{X}\to \overline{X}$ and $g\colon \tilde{Z}\to Z$ are birational morphisms, 
\item
$\tilde{\pi}\colon \tilde{X}\to \tilde{Z}$ is a contraction and all fibers of $\tilde{\pi}$ have the same dimensions, and 
\item 
$(\tilde{X},0)$ is $\mathbb{Q}$-factorial klt and $\bigl(\tilde{X},{\rm Supp}\tilde{f}_{*}^{-1}\overline{B}\cup {\rm Ex}(\tilde{f})\bigr)$ is an lc pair, here we regard ${\rm Supp}\tilde{f}_{*}^{-1}\overline{B}\cup {\rm Ex}(\tilde{f})$ as a reduced divisor on $\tilde{X}$. 
\end{itemize}
By the third condition, there is an $\mathbb{R}$-divisor $\tilde{B}$ on $\tilde{X}$ such that $(\tilde{X},\tilde{B})$ is a $\mathbb{Q}$-factorial klt pair and $K_{\tilde{X}}+\tilde{B}-\tilde{f}^{*}(K_{\overline{X}}+\overline{B})$ is effective and $\tilde{f}$-exceptional. 
Set $f=\overline{f}\circ \tilde{f}\colon \tilde{X} \to X$. 
Because $\boldsymbol{\rm M}=\overline{\boldsymbol{\rm M}_{\overline{X}}}$ and $\boldsymbol{\rm M}$ is b-nef, $\boldsymbol{\rm M}_{\tilde{X}}$ is nef, so $\bigl(\tilde{X},\tilde{B}, \boldsymbol{\rm M}\bigr)$ is a generalized klt pair.  
Furthermore, by construction of $\overline{B}$ and $\tilde{B}$, we may write 
$$K_{\tilde{X}}+\tilde{B}+\boldsymbol{\rm M}_{\tilde{X}}=f^{*}(K_{X}+B+\boldsymbol{\rm M}_{X})+\tilde{E}$$
for some $f$-exceptional $\mathbb{R}$-divisor $\tilde{E}\geq0$. 

In this way, we obtain a diagram
 $$
\xymatrix{
X \ar[d]_{\pi}&\tilde{X}\ar[l]_{f}\ar[d]^{\tilde{\pi}}\\
Z&\tilde{Z}\ar[l]^{g}
}
$$
and a projective $\mathbb{Q}$-factorial generalized klt pair $\bigl(\tilde{X},\tilde{B}, \boldsymbol{\rm M}\bigr)$ with $\boldsymbol{\rm M}=\overline{\boldsymbol{\rm M}_{\tilde{X}}}$ such that 
\begin{enumerate}
\item \label{proof--prop--genabklt-(1)}
$f \colon \tilde{X}\to \overline{X}$ and $g\colon \tilde{Z}\to Z$ are birational morphisms and $\tilde{Z}$ is smooth, 
\item \label{proof--prop--genabklt-(2)}
$\tilde{\pi}\colon \tilde{X}\to \tilde{Z}$ is a contraction and all fibers of $\tilde{\pi}$ have the same dimensions, and 
\item \label{proof--prop--genabklt-(3)}
$K_{\tilde{X}}+\tilde{B}+\boldsymbol{\rm M}_{\tilde{X}}=f^{*}(K_{X}+B+\boldsymbol{\rm M}_{X})+\tilde{E}$ for some $f$-exceptional $\mathbb{R}$-divisor $\tilde{E}\geq0$. 
\end{enumerate}
By property  (\ref{proof--prop--genabklt-(3)}), it follows that $K_{X}+B+\boldsymbol{\rm M}_{X}+\pi^{*}A_{Z}$ is abundant and birationally has the Nakayama--Zariski decomposition with semi-ample positive part if and only if  the same assertion holds for $K_{\tilde{X}}+\tilde{B}+\boldsymbol{\rm M}_{\tilde{X}}+\tilde{\pi}^{*}g^{*}A_{Z}$. 
\end{step3}

\begin{step3}\label{step2genabklt}
By (\ref{proof--prop--genabklt-(3)}) in Step \ref{step1genabklt} and the hypothesis that $(K_{X}+B+\boldsymbol{\rm M}_{X})|_{\pi^{-1}(U)}\sim_{\mathbb{R},U}0$ for some open $U\subset Z$, there exist an $\mathbb{R}$-divisor $\tilde{F}\geq0$ on $\tilde{X}$ such that $K_{\tilde{X}}+\tilde{B}+\boldsymbol{\rm M}_{\tilde{X}}\sim_{\mathbb{R},\tilde{Z}}\tilde{F}$. 
We can write $\tilde{F}=\tilde{F}_{h}+\tilde{F}_{v}$ such that all components of $\tilde{F}_{h}$ dominates $\tilde{Z}$ and all components $\tilde{F}_{v}$ are vertical over $\tilde{Z}$. 
We recall condition (\ref{proof--prop--genabklt-(2)}) in Step \ref{step1genabklt}, which says that all fibers of $\tilde{\pi}$ have the same dimensions.  
Therefore, the image of any component of $\tilde{F}_{v}$ on $\tilde{Z}$ is a divisor. 
Since $\tilde{Z}$ is smooth, we can consider 
$$\nu_{P}:={\rm sup}\set{\nu\in\mathbb{R}_{\geq0} \mid \text{$\tilde{F}_{v}-\nu \tilde{\pi}^{*}P\geq0$}}$$
for any prime divisor $P$ on $\tilde{Z}$. 
Then it is easy to see that $\nu_{P}>0$ with only finitely many prime divisors $P$. 
By (\ref{proof--prop--genabklt-(2)}) in Step \ref{step1genabklt}, we see that $\tilde{F}_{v}-\sum_{P}\nu_{P} \tilde{\pi}^{*}P$ is an effective $\mathbb{R}$-divisor which is very exceptional over $\tilde{Z}$. 
By replacing $\tilde{F}_{v}$ with $\tilde{F}_{v}-\sum_{P}\nu_{P} \tilde{\pi}^{*}P$, we may assume that $\tilde{F}_{v}$ is very exceptional over $\tilde{Z}$. 

We run a $(K_{\tilde{X}}+\tilde{B}+\boldsymbol{\rm M}_{\tilde{X}})$-MMP over $\tilde{Z}$ with scaling of an ample divisor 
$$(\tilde{X},\tilde{B},\boldsymbol{\rm M})=(\tilde{X}_{0},\tilde{B}_{0},\boldsymbol{\rm M})\dashrightarrow \cdots \dashrightarrow (\tilde{X}_{i},\tilde{B}_{i},\boldsymbol{\rm M})\dashrightarrow \cdots.$$
We prove that after finitely many steps we reach a model $(\tilde{X}_{m},\tilde{B}_{m},\boldsymbol{\rm M})$ such that $K_{\tilde{X}_{m}}+\tilde{B}_{m}+\boldsymbol{\rm M}_{\tilde{X}_{m}}\sim_{\mathbb{R},\tilde{Z}}0$. 
We recall the hypothesis that $(K_{X}+B+\boldsymbol{\rm M}_{X})|_{\pi^{-1}(U)}\sim_{\mathbb{R},U}0$ for some open $U\subset Z$. 
By shrinking $U$, we may assume $g\colon \tilde{Z}\to Z$ is an isomorphism over $U$,  so $U$ can be thought of an open subset of $\tilde{Z}$. 
For each $i$, let $\tilde{V}_{i}$ be the inverse image of $U$ to $\tilde{X}_{i}$. 
Then, we see that 
$$(\tilde{V}_{0},\tilde{B}_{0}|_{\tilde{V}_{0}}, \boldsymbol{\rm M}|_{\tilde{V}_{0}})\dashrightarrow \cdots \dashrightarrow (\tilde{V}_{i},\tilde{B}_{i}|_{\tilde{V}_{i}}, \boldsymbol{\rm M}|_{\tilde{V}_{i}})\dashrightarrow \cdots$$ 
is a sequence of steps of a $(K_{\tilde{V}_{0}}+\tilde{B}_{0}|_{\tilde{V}_{0}}+\boldsymbol{\rm M}_{\tilde{X}_{0}}|_{\tilde{V}_{0}})$-MMP over $U$, and the generalized klt pair $(\tilde{V}_{0},\tilde{B}_{0}|_{\tilde{V}_{0}},\boldsymbol{\rm M}|_{\tilde{V}_{0}})$ has a relatively trivial minimal model over $U$ because of (\ref{proof--prop--genabklt-(3)}) in Step \ref{step1genabklt} and the relation $(K_{X}+B+\boldsymbol{\rm M}_{X})|_{\pi^{-1}(U)}\sim_{\mathbb{R},U}0$. 
By \cite[Theorem 4.1]{hanli}, for any $i\gg0$ we obtain $K_{\tilde{V}_{i}}+\tilde{B}_{i}|_{\tilde{V}_{i}}+\boldsymbol{\rm M}_{\tilde{X}_{i}}|_{\tilde{V}_{i}}\sim_{\mathbb{R},U}0$. 
Because we have
$K_{\tilde{X}}+\tilde{B}+\boldsymbol{\rm M}_{\tilde{X}}\sim_{\mathbb{R},\tilde{Z}}\tilde{F}_{h}+\tilde{F}_{v}$ and each component of $\tilde{F}_{h}$ dominates over $\tilde{Z}$, 
we see that $\tilde{F}_{h}$ is contracted by the map $\tilde{X}\dashrightarrow \tilde{X}_{i}$ for any $i\gg0$. 
Pick $m>0$ such that $\tilde{F}_{h}$ is contracted by $\tilde{X}\dashrightarrow \tilde{X}_{m}$. 
Since the  $(K_{\tilde{X}}+\tilde{B}+\boldsymbol{\rm M}_{\tilde{X}})$-MMP over $\tilde{Z}$ occurs only in ${\rm Supp}(\tilde{F}_{h}+\tilde{F}_{v})$, the birational transform of $\tilde{F}_{v}$ on $\tilde{X}_{m}$ is very exceptional over $\tilde{Z}$. 
Applying \cite[Proposition 3.8]{hanli} to $\bigl(\tilde{X}_{m}, \tilde{B}_{m},\boldsymbol{\rm M}\bigr)/\tilde{Z}$ and replacing $m$, we see that $\tilde{F}_{v}$ is also contracted by $\tilde{X}\dashrightarrow \tilde{X}_{m}$. 

From this discussion, there exists $m>0$ such that $K_{\tilde{X}_{m}}+\tilde{B}_{m}+\boldsymbol{\rm M}_{\tilde{X}_{m}}\sim_{\mathbb{R},\tilde{Z}}0$. 
\end{step3}

\begin{step3}\label{step3genabklt} 
We denote the natural murphism $\tilde{X}\to \tilde{Z}$ by $\tilde{\pi}_{m}$. 
Now we have the following diagram
 $$
\xymatrix{
X \ar[d]_{\pi}&\tilde{X}\ar[l]_{f}\ar[d]^{\tilde{\pi}}\ar@{-->}[r]&\tilde{X}_{m}\ar[dl]^{\tilde{\pi}_{m}}\\
Z&\tilde{Z}\ar[l]^{g}
}
$$
and a projective generalized klt pair $\bigl(\tilde{X}_{m}, \tilde{B}_{m},\boldsymbol{\rm M}\bigr)$ such that $K_{\tilde{X}_{m}}+\tilde{B}_{m}+\boldsymbol{\rm M}_{\tilde{X}_{m}}\sim_{\mathbb{R},\tilde{Z}}0$. 
Since $\tilde{X}\dashrightarrow \tilde{X}_{m}$  is also a sequence of steps of a $(K_{\tilde{X}}+\tilde{B}+\tilde{\pi}^{*}g^{*}A_{Z}+\boldsymbol{\rm M}_{\tilde{X}})$-MMP,  the divisor $K_{\tilde{X}}+\tilde{B}+\boldsymbol{\rm M}_{\tilde{X}}+\tilde{\pi}^{*}g^{*}A_{Z}$ is abundant and birationally has the Nakayama--Zariski decomposition with semi-ample positive part if and only if the same assertion holds for $K_{\tilde{X}_{m}}+\tilde{B}_{m}+\tilde{\pi}_{m}^{*}g^{*}A_{Z}+\boldsymbol{\rm M}_{\tilde{X}_{m}}$. 

We recall the first condition of Proposition \ref{prop--genabklt} saying that $\boldsymbol{\rm M}$ is a finite $\mathbb{R}_{>0}$-linear combination of b-nef $\mathbb{Q}$-b-Cartier $\mathbb{Q}$-b-divisors. 
By an argument using convex geometry, we can find finitely many positive real numbers $r_{1},\cdots, r_{l}$ and projective generalized klt pairs $\bigl(\tilde{X}_{m}, \tilde{B}^{(1)}_{m},\boldsymbol{\rm M}^{(1)}\bigr), \cdots, \bigl(\tilde{X}_{m}, \tilde{B}^{(l)}_{m},\boldsymbol{{\rm M}}^{(l)}\bigr) $ such that 
\begin{itemize}
\item
$\tilde{B}^{(j)}_{m}$ is a $\mathbb{Q}$-divisors and $\boldsymbol{\rm M}^{(j)}$ is a $\mathbb{Q}$-b-divisors for every $1\leq j \leq l$, 
\item
$\sum_{j=1}^{l}r_{j}=1$, $\sum_{j=1}^{l}r_{j}\tilde{B}^{(j)}_{m}=\tilde{B}_{m}$ and $\sum_{j=1}^{l}r_{j}\boldsymbol{{\rm M}}^{(j)}=\boldsymbol{\rm M}$, and
\item
$K_{\tilde{X}_{m}}+\tilde{B}^{(j)}_{m}+\boldsymbol{\rm M}^{(j)}\sim_{\mathbb{Q},\tilde{Z}}0$ for any $1\leq j \leq l$. 
\end{itemize}
By applying the canonical bundle formula for generalized pairs \cite[Theorem 1.4]{filcanbundleformula} to each $\bigl(\tilde{X}_{m}, \tilde{B}^{(j)}_{m},\boldsymbol{{\rm M}}^{(j)}\bigr)\to \tilde{Z}$ and using the above conditions, we can find a generalized klt pair $(\tilde{Z},B_{\tilde{Z}},\boldsymbol{\rm N})$ such that $K_{\tilde{X}_{m}}+\tilde{B}_{m}+\boldsymbol{\rm M}_{\tilde{X}_{m}}\sim_{\mathbb{R}}\tilde{\pi}_{m}^{*}(K_{\tilde{Z}}+B_{\tilde{Z}}+\boldsymbol{\rm N}_{\tilde{Z}})$. 
Since $g^{*}A_{Z}$ is nef and big, we can find a big $\mathbb{R}$-divisor $\Delta_{\tilde{Z}}$ on $\tilde{Z}$ such that $K_{\tilde{Z}}+B_{\tilde{Z}}+\boldsymbol{\rm N}_{\tilde{Z}}+g^{*}A_{Z}\sim_{\mathbb{R}}K_{\tilde{Z}}+\Delta_{\tilde{Z}}$ and $(\tilde{Z},\Delta_{\tilde{Z}})$ is klt. 
By \cite{bchm}, $(\tilde{Z},\Delta_{\tilde{Z}})$ has a good minimal model or a Mori fiber space. 
In particular, $K_{\tilde{Z}}+\Delta_{\tilde{Z}}$ is abundant, and the divisor birationally has the Nakayama--Zariski decomposition with semi-ample positive part when it is pseudo-effective. 
Because 
$$K_{\tilde{X}_{m}}+\tilde{B}_{m}+\tilde{\pi}_{m}^{*}g^{*}A_{Z}+\boldsymbol{\rm M}_{\tilde{X}_{m}}\sim_{\mathbb{R}}\tilde{\pi}_{m}^{*}(K_{\tilde{Z}}+\Delta_{\tilde{Z}}),$$
we see that $K_{\tilde{X}_{m}}+\tilde{B}_{m}+\tilde{\pi}_{m}^{*}g^{*}A_{Z}+\boldsymbol{\rm M}_{\tilde{X}_{m}}$ is abundant.
When $K_{\tilde{X}_{m}}+\tilde{B}_{m}+\tilde{\pi}_{m}^{*}g^{*}A_{Z}+\boldsymbol{\rm M}_{\tilde{X}_{m}}$ is pseudo-effective, applying \cite[Corollary 5.17]{nakayama}, we see that $K_{\tilde{X}_{m}}+\tilde{B}_{m}+\tilde{\pi}_{m}^{*}g^{*}A_{Z}+\boldsymbol{\rm M}_{\tilde{X}_{m}}$ birationally has the Nakayama--Zariski decomposition with semi-ample positive part. 

\end{step3}
By the first paragraph of Step \ref{step3genabklt} and the final sentence of Step \ref{step1genabklt}, 
$K_{X}+B+\boldsymbol{\rm M}_{X}+\pi^{*}A_{Z}$ is abundant, and the divisor birationally has the Nakayama--Zariski decomposition with semi-ample positive part when it is pseudo-effective. 
We finish the proof. 
\end{proof}

\begin{thm}\label{thm--abund-sub}
Let $\pi\colon X\to Z$ be a morphism of normal projective varieties, and let $(X,B, \boldsymbol{\rm M})$ be a generalized lc pair. 
Suppose that  
\makeatletter 
\renewcommand{\p@enumii}{II-} 
\makeatother
\begin{enumerate}[(I)]
\item \label{thm--abund-sub-(I)}
$\boldsymbol{\rm M}$ is a finite $\mathbb{R}_{>0}$-linear combination of b-nef $\mathbb{Q}$-b-Cartier $\mathbb{Q}$-b-divisors, and  
\item \label{thm--abund-sub-(II)}
there is an effective $\mathbb{R}$-Cartier divisor $C$ on $X$ such that
\begin{enumerate}[({II-}a)]
\item \label{thm--abund-sub-(II-a)}
the generalized pair $(X,B+tC, \boldsymbol{\rm M})$ is generalized lc for some $t>0$, and  
\item \label{thm--abund-sub-(II-b)}
$K_{X}+B+C+\boldsymbol{\rm M}_{X}\sim_{\mathbb{R},Z}0$.
\end{enumerate}
\end{enumerate}
Let $A_{Z}$ be an ample $\mathbb{R}$-divisor on $Z$, and pick $0\leq A\sim_{\mathbb{R}}\pi^{*}A_{Z}$ such that $(X,B+A, \boldsymbol{\rm M})$ is a generalized lc pair. 

Then, $K_{X}+B+A+\boldsymbol{\rm M}_{X}$ is abundant. 
\end{thm}

Before proving the theorem, we prove the following lemma. 

\begin{lem}\label{lem--abund-birat}
Assume Theorem \ref{thm--abund-sub} for all projective generalized lc pairs of dimension at most $n-1$. 
Let $\pi\colon X\to Z$ be a morphism, $(X,B, \boldsymbol{\rm M})$ a projective generalized lc pair, and let $C$, $A_{Z}$ and $A$ be $\mathbb{R}$-Cartier divisors as in Theorem \ref{thm--abund-sub} such that ${\rm dim}X\leq n- 1$. 
Let $f\colon \tilde{X} \to X$ be a log resolution of $(X,{\rm Supp}B)$ such that $\boldsymbol{\rm M}=\overline{\boldsymbol{\rm M}_{\tilde{X}}}$. 
Let $\tilde{B}\geq0$ be an $\mathbb{R}$-divisor on $\tilde{X}$ such that the pair $(\tilde{X},\tilde{B}+f^{*}A)$ is an lc pair and the effective part of the divisor $K_{\tilde{X}}+\tilde{B}+\boldsymbol{\rm M}_{\tilde{X}}-f^{*}(K_{X}+B+\boldsymbol{\rm M}_{X})$ is $f$-exceptional. 

Then, $K_{\tilde{X}}+\tilde{B}+f^{*}A+\boldsymbol{\rm M}_{\tilde{X}}$ is abundant. 
\end{lem}

\begin{proof}
We set $\tilde{A}=f^{*}A$. 
We can write
$$K_{\tilde{X}}+\tilde{B}+\boldsymbol{\rm M}_{\tilde{X}}-f^{*}(K_{X}+B+\boldsymbol{\rm M}_{X})+\tilde{E}_{+}-\tilde{E}_{-}$$
where $\tilde{E}_{+}\geq0$ and $\tilde{E}_{-}\geq0$ are $f$-exceptional $\mathbb{R}$-divisors having no common components. 
We run a $(K_{\tilde{X}}+\tilde{B}+\tilde{A}+\boldsymbol{\rm M}_{\tilde{X}})$-MMP over $X$ with scaling of an ample divisor. 
By applying \cite[Proposition 3.8]{hanli} to $(\tilde{X},\tilde{B}+\tilde{A},\boldsymbol{\rm M})/X$, after finitely many steps we get a projective morphism 
$f'\colon \tilde{X}' \to X$ such that 
$$K_{\tilde{X}'}+\tilde{B}'+\tilde{A}'+\boldsymbol{\rm M}_{\tilde{X}'}=f'^{*}(K_{X}+B+\boldsymbol{\rm M}_{X}+A)-\tilde{E}'_{-},$$
where $\tilde{B}'$, $\tilde{A}'$, and $\tilde{E}'_{-}$ are the birational transforms of $\tilde{B}$, $\tilde{A}$, and $\tilde{E}_{-}$ on $\tilde{X}'$, respectively. 
Then $\tilde{A}'=f'^{*}A$, and we have
$$K_{\tilde{X}'}+\tilde{B}'+(u f'^{*}C+\tilde{E}'_{-})+\boldsymbol{\rm M}_{\tilde{X}'}=f'^{*}(K_{X}+B+u C+\boldsymbol{\rm M}_{X})$$ 
for all real numbers $u\geq 0$. 
By (\ref{thm--abund-sub-(II-a)}) in Theorem \ref{thm--abund-sub}, we see that the generalized pair $\bigl(\tilde{X'}, \tilde{B}'+t f'^{*}C+\tilde{E}'_{-}, \boldsymbol{\rm M}\bigr)$
is generalized lc for some $t>0$. 
Putting $t'={\rm min}\{1,t\}$, we see that $\bigl(\tilde{X'}, \tilde{B}'+t'(f'^{*}C+\tilde{E}'_{-}), \boldsymbol{\rm M}\bigr)$ is a generalized lc pair. 
Furthermore, by (\ref{thm--abund-sub-(II-b)}) in Theorem \ref{thm--abund-sub}, we obtain the relation
$$K_{\tilde{X}'}+\tilde{B}'+(f'^{*}C+\tilde{E}'_{-})+\boldsymbol{\rm M}_{\tilde{X}'}=f'^{*}(K_{X}+B+C+\boldsymbol{\rm M}_{X})\sim_{\mathbb{R},Z}0.$$ 
By construction, it is obvious that $\bigl(\tilde{X}',\tilde{B}'+\tilde{A}',\boldsymbol{\rm M}\bigr)$ is a generalized lc pair. 
From these arguments, we can apply Theorem \ref{thm--abund-sub} to $\pi\circ f'\colon \tilde{X}\to Z$, $\bigl(\tilde{X'}, \tilde{B}', \boldsymbol{\rm M}\bigr)$, $f'^{*}C+\tilde{E}'_{-}$, $A_{Z}$ and $\tilde{A}'$. 
We see that $K_{\tilde{X}'}+\tilde{B}'+\tilde{A}'+\boldsymbol{\rm M}_{\tilde{X}'}$ is abundant. 
Since $\tilde{X}'$ is constructed by running a $(K_{\tilde{X}}+\tilde{B}+\tilde{A}+\boldsymbol{\rm M}_{\tilde{X}})$-MMP, the divisor $K_{\tilde{X}}+\tilde{B}+\tilde{A}+\boldsymbol{\rm M}_{\tilde{X}}$ is also abundant. 
\end{proof}

From now on, we prove Theorem \ref{thm--abund-sub}. 

\begin{proof}[Proof of Theorem \ref{thm--abund-sub}]
The argument is very similar to \cite[Proof of Theorem 5.4]{hashizumehu}. 
We prove Theorem \ref{thm--abund-sub} by induction on the dimension of $X$. 
Assume Theorem \ref{thm--abund-sub} for all projective generalized lc pairs of dimension $\leq n-1$. 
Let $\pi\colon X\to Z$ be a projective morphism, $(X,B,\boldsymbol{\rm M})$ a projective generalized lc pair, and let $C$, $A_{Z}$ and $A$ be $\mathbb{R}$-Cartier divisors as in Theorem \ref{thm--abund-sub} such that ${\rm dim}X= n$. 
We may assume that $K_{X}+B+A+\boldsymbol{\rm M}_{X}$ is pseudo-effective. 
By replacing $A$ with a general one, we may assume that $B$ and $A$ have no common components and all generalized lc centers of $(X,B+A,\boldsymbol{\rm M})$ are those of $(X,B,\boldsymbol{\rm M})$. 
By replacing $(X,B,\boldsymbol{\rm M})$ with a generalized $\mathbb{Q}$-factorial dlt model and replacing $C$ and $A$ with the pullbacks, we may assume that $(X,B,\boldsymbol{\rm M})$ is a generalized $\mathbb{Q}$-factorial dlt pair. 
Note that the conditions (\ref{thm--abund-sub-(I)}), (\ref{thm--abund-sub-(II-a)}), and (\ref{thm--abund-sub-(II-b)}) still hold after replacing $(X,B,\boldsymbol{\rm M})$, $C$, and $A$. 

From now on, we divide the proof into several steps. 
Each step corresponds to that of \cite[Proof of Theorem 5.4]{hashizumehu}. 

\begin{step4}\label{step1abund}
In this step, we prove Theorem \ref{thm--abund-sub} in the case where $K_{X}+B-\epsilon\llcorner B\lrcorner +A+\boldsymbol{\rm M}_{X}$ is pseudo-effective for some real number $\epsilon>0$. 
Note that this case includes the generalized klt case. 

Since we have $K_{X}+B+C+\boldsymbol{\rm M}_{X}\sim_{\mathbb{R},Z}0$, by restricting $K_{X}+B-\epsilon\llcorner B\lrcorner +A+\boldsymbol{\rm M}_{X}$ to a  general fiber of the Stein factorization of $\pi$, we see that $\llcorner B\lrcorner$ and $C$ are vertical over $Z$. 
Then there is an open subset $U\subset Z$ such that 
$(K_{X}+B-\epsilon\llcorner B\lrcorner +\boldsymbol{\rm M}_{X})|_{\pi^{-1}(U)}\sim_{\mathbb{R},U}0$. 
By Proposition \ref{prop--genabklt}, $K_{X}+B-\epsilon\llcorner B\lrcorner +A+\boldsymbol{\rm M}_{X}$ is abundant, so $K_{X}+B-\epsilon\llcorner B\lrcorner +A+\boldsymbol{\rm M}_{X}\sim_{\mathbb{R}}G$ for an $\mathbb{R}$-divisor $G\geq0$. 
Similarly, we see that $K_{X}+B-\tfrac{\epsilon}{2}\llcorner B\lrcorner +A+\boldsymbol{\rm M}_{X}$ is abundant. 
Now we have
\begin{equation*}
\begin{split}
K_{X}+B-\tfrac{\epsilon}{2}\llcorner B\lrcorner +A+\boldsymbol{\rm M}_{X}\sim_{\mathbb{R}}G+\tfrac{\epsilon}{2}\llcorner B\lrcorner \quad {\rm and}\quad K_{X}+B+A+\boldsymbol{\rm M}_{X}\sim_{\mathbb{R}}G+\llcorner B\lrcorner. 
\end{split}
\end{equation*}
By \cite[Remark 2.15]{hashizumehu}, we have
\begin{equation*}
\begin{split}
\kappa_{\sigma}(X,K_{X}+B+A+\boldsymbol{\rm M}_{X})&=\kappa_{\sigma}(X,K_{X}+B-\tfrac{\epsilon}{2}\llcorner B\lrcorner +A+\boldsymbol{\rm M}_{X})\\
&=\kappa_{\iota}(X,K_{X}+B-\tfrac{\epsilon}{2}\llcorner B\lrcorner +A+\boldsymbol{\rm M}_{X})\\
&=\kappa_{\iota}(X,K_{X}+B+A+\boldsymbol{\rm M}_{X}).
\end{split}
\end{equation*}
In this way, we see that $K_{X}+B+A+\boldsymbol{\rm M}_{X}$ is abundant. 
\end{step4}

\begin{step4}\label{step2abund}
By Step \ref{step1abund}, we may assume that the divisor $K_{X}+B-\epsilon\llcorner B\lrcorner +A+\boldsymbol{\rm M}_{X}$ is not pseudo-effective for any real number $\epsilon>0$. 
Then there exists a component $S$ of $\llcorner B\lrcorner$ such that $K_{X}+B-\epsilon S+A+\boldsymbol{\rm M}_{X}$ is not pseudo-effective for any real number $\epsilon>0$. 
In this step, we will apply \cite[Proof of Lemma 4.4]{hanliu-nonvanish} (see also \cite[Lemma 2.9]{birkar-connectedness}). 

For any $\epsilon'\in(0,1]$, by running a $(K_{X}+B-\epsilon'S +A+\boldsymbol{\rm M}_{X})$-MMP with scaling of an ample divisor, we obtain a birational contraction $X\dashrightarrow X'$ and a projective generalized lc pair $(X',B'-\epsilon'S'+A',\boldsymbol{\rm M})$ admitting a structure of Mori fiber space $X'\to Z'$, that is, ${\rm dim}Z'<{\rm dim}X'$, $-(K_{X'}+B'-\epsilon'S'+A'+\boldsymbol{\rm M}_{X'})$ is ample over $Z'$, and the relative Picard number is $1$. 
With the ACC for generalized lc thresholds (\cite[Theorem 1.5]{bz}), we can find $\epsilon_{0}\in(0,1]$ such that if $\epsilon'\in(0,\epsilon_{0})$, then $(X',B'+A',\boldsymbol{\rm M})$ is generalized lc. 

By (\ref{thm--abund-sub-(I)}) of Theorem \ref{thm--abund-sub}, we may write $\boldsymbol{\rm M}=\sum_{i=1}^{l}r_{l}\boldsymbol{\rm M}^{(i)}$ where $r_{i}$ are positive real numbers and $\boldsymbol{\rm M}^{(i)}$ are b-nef $\mathbb{Q}$-b-Cartier $\mathbb{Q}$-b-divisors. 
Since $X'$ is $\mathbb{Q}$-factorial, $\boldsymbol{\rm M}^{(i)}_{X'}$ are $\mathbb{Q}$-Cartier divisors on $X'$. 
By construction, we may find a positive real number $t_{\epsilon'}<\epsilon'$ such that $K_{X'}+B'-t_{\epsilon'}S'+A'+\boldsymbol{\rm M}_{X'}$ is numerically trivial over $Z'$. 
We pick a general fiber $F'$ of $X'\to Z'$. 
Then $K_{F'}+(B'-t_{\epsilon'}S'+A'+\boldsymbol{\rm M}_{X'})|_{F'}\equiv0$ and $\boldsymbol{\rm M}|_{F'}=\sum_{i=1}^{l}r_{l}\boldsymbol{\rm M}^{(i)}|_{F'}$. 
Here, $\boldsymbol{\rm M}^{(i)}|_{F'}$ are b-nef $\mathbb{Q}$-b-Cartier $\mathbb{Q}$-b-divisors on $F'$ and the trace of each $\boldsymbol{\rm M}^{(i)}|_{F'}$ on $F'$ is a $\mathbb{Q}$-Cartier divisor. 
By the negativity lemma, for each $i$, if $\boldsymbol{\rm M}^{(i)}|_{F'}$ is the closure of a numerically trivial $\mathbb{Q}$-Cartier divisor then the trace of $\boldsymbol{\rm M}^{(i)}|_{F'}$ on $F'$ is numerically trivial and $\boldsymbol{\rm M}^{(i)}|_{F'}$ descends to $F'$. 
After discarding $\boldsymbol{\rm M}^{(i)}|_{F'}$ which are the closures of  numerically trivial divisors (such $\boldsymbol{\rm M}^{(i)}|_{F'}$ depend on $\epsilon'$), we can apply the global ACC \cite[Theorem 1.6]{bz} to the set $\bigl\{\bigl(F', (B'-t_{\epsilon'}S'+A')|_{F'}, \boldsymbol{\rm M}_{X'}|_{F'}\bigr)\,\big|\, \epsilon'\in (0,1]\bigr\}$. 
We can find $\epsilon'\in(0,\epsilon_{0})$ such that $X\dashrightarrow X'$ and $X'\to Z'$ as above satisfy that $(X',B'+A',\boldsymbol{\rm M})$ is generalized lc and $K_{X'}+B'+A'+\boldsymbol{\rm M}_{X'}\sim_{\mathbb{R},Z'}0$. 

By the above discussions, by choosing $\epsilon'\in(0,1]$ sufficiently small, we get a birational contraction $X\dashrightarrow X'$ and a contraction $X'\to Z'$ such that $(X',0)$ is $\mathbb{Q}$-factorial klt, $(X',B'+A',\boldsymbol{\rm M})$ is generalized lc, and $K_{X'}+B'+A'+\boldsymbol{\rm M}_{X'}\sim_{\mathbb{R},Z'}0$. 
\end{step4}

\begin{step4}\label{step3abund}
The goal of this step is to construct a diagram 
$$
\xymatrix
{
\tilde{X}\ar[d]_{f}\ar@{-->}[r]&\tilde{X}'\ar[d]\\
X&\tilde{Z}
}
$$
having good properties, where $f$ is a log resolution of $(X,{\rm Supp}B)$ such that $\boldsymbol{\rm M}=\overline{\boldsymbol{\rm M}_{\tilde{X}}}$, the map $\tilde{X}\dashrightarrow \tilde{X}'$ is a sequence of steps of an MMP over $Z'$ of a divisor $K_{\tilde{X}}+\tilde{B}+\tilde{A}+\boldsymbol{\rm M}_{\tilde{X}}$ such that the property of being abundant for $K_{X}+B+A+\boldsymbol{\rm M}_{X}$ follows from that for $K_{\tilde{X}}+\tilde{B}+\tilde{A}+\boldsymbol{\rm M}_{\tilde{X}}$, and $\tilde{X}'\to \tilde{Z}$ is a contraction over $Z'$ induced by the birational transform of $K_{\tilde{X}}+\tilde{B}+\tilde{A}+\boldsymbol{\rm M}_{\tilde{X}}$ on $\tilde{X}'$. 

Let $\overline{f}\colon \overline{X} \to X$ and $\overline{f}'\colon \overline{X}\to X'$ be a common resolution of $X \dashrightarrow X'$ such that $\overline{f}$ and $\overline{f}'$ are log resolutions of $(X,{\rm Supp}B)$ and $(X',{\rm Supp}B')$ respectively and $\boldsymbol{\rm M}=\overline{\boldsymbol{\rm M}_{\overline{X}}}$. 
Since $A$ is general, we may replace $A$ by a general member of its $\mathbb{R}$-linear system without changing the coefficients. 
The ACC for generalized lc thresholds guarantees that the property of $(X',B'+A',\boldsymbol{\rm M})$ being generalized lc is preserved after we replace $A$. 
Thus, we can replace $A$ with a general one keeping that $(X',B'+A',\boldsymbol{\rm M})$ is generalized lc and $K_{X'}+B'+A'+\boldsymbol{\rm M}_{X'}\sim_{\mathbb{R},Z'}0$. 
By \cite[Lemma 2.2]{hashizumehu} and replacing $A$, we may assume that $\overline{f}$ (resp.~$\overline{f}'$) is a log resolution of $(X,{\rm Supp}(B+A))$ (resp.~$(X',{\rm Supp}(B'+A'))$) and that ${\rm Supp}\overline{f}^{*}A$ and ${\rm Supp}\overline{f}_{*}^{-1}B\cup {\rm Ex}(\overline{f})\cup {\rm Ex}(\overline{f}')$ have no common divisorial components. 
Then $\overline{f}^{*}A=\overline{f}_{*}^{-1}A$. 
Putting $\overline{A}=\overline{f}^{*}A$, we can write 
\begin{equation*}\tag{1}\label{proof-thm--abund-sub-(1)}
K_{\overline{X}}+\overline{B}+\overline{A}+\boldsymbol{\rm M}_{\overline{X}}=\overline{f}^{*}(K_{X}+B+A+\boldsymbol{\rm M}_{X})+\overline{E}
\end{equation*}
with $\overline{B}\geq0$ and $\overline{E}\geq0$ which have no common components. 
Then $\overline{B}+\overline{A}$ and $\overline{E}$ have no common components and $(\overline{X},\overline{B}+\overline{A},\boldsymbol{\rm M})$ is a $\mathbb{Q}$-factorial generalized dlt pair with the nef part $\boldsymbol{\rm M}=\overline{\boldsymbol{\rm M}_{\overline{X}}}$. 

We check that the generalized lc pair $(X',B'+A',\boldsymbol{\rm M})/Z$, the morphisms $\overline{f}'\colon \overline{X}\to X'$ and  $X'\to Z'$, and the generalized lc pair $(\overline{X},\overline{B}+\overline{A},\boldsymbol{\rm M})$ satisfy all the conditions of Proposition \ref{prop--relmmp}. 
Note that we may regard $(X',B'+A',\boldsymbol{\rm M})$ as a generalized lc pair over $Z$ because $\boldsymbol{\rm M}$ is b-nef. 
By construction, $\overline{f}'$ is a log resolution of $(X',{\rm Supp}(B'+A'))$ and $X'\to Z'$ is a contraction. 
Since $K_{X}+B+A+\boldsymbol{\rm M}_{X}$ is pseudo-effective, $K_{\overline{X}}+\overline{B}+\overline{A}+\boldsymbol{\rm M}_{\overline{X}}$ is also pseudo-effective, which is the first condition of Proposition \ref{prop--relmmp}. 
The relation $K_{X'}+B'+A'+\boldsymbol{\rm M}_{X'}\sim_{\mathbb{R},Z'}0$ is nothing but the second condition of Proposition \ref{prop--relmmp}. 
By construction of $X\dashrightarrow X'$ and $X'\to Z$, the divisor $S'$ is ample over $Z'$ and $(X',0)$ is a klt pair. 
Since $X'$ is $\mathbb{Q}$-factorial, there is $\Delta'$ on $X'$ such that $(X',\Delta')$ is klt and $K_{X'}+\Delta'\sim_{\mathbb{R},Z'}K_{X'}+B'+A'+\boldsymbol{\rm M}_{X'}\sim_{\mathbb{R},Z'}0$. 
Thus, we can apply Proposition \ref{prop--relmmp} to the morphisms $(\overline{X},\overline{B}+\overline{A},\boldsymbol{\rm M})\to X'$ and $X'\to Z'$, amd there is a birational contraction $\overline{X}\dashrightarrow \overline{X}'$ over $Z'$ such that if we set $\overline{B}'$ and $\overline{A}'$ as the birational transforms of $\overline{B}$ and $\overline{A}$ on $\overline{X}'$ respectively, then we have
\begin{enumerate}[(i)]
\item \label{proof-thm--abund-sub-(i)}
$K_{\overline{X}'}+\overline{B}'+\overline{A}'+\boldsymbol{\rm M}_{\overline{X}'}$ is semi-ample over $Z'$, and
\item \label{proof-thm--abund-sub-(ii)}
$a(P,\overline{X},\overline{B}+\overline{A}+\boldsymbol{\rm M}_{\overline{X}})\leq a(P,\overline{X}',\overline{B}'+\overline{A}'+\boldsymbol{\rm M}_{\overline{X}'})$ for any prime divisor $P$ over $\overline{X}$. 
\end{enumerate}
Now we have the following diagram.
$$
\xymatrix@R=15pt
{
\overline{X}\ar[d]_{\overline{f}}\ar@{-->}[r]&\overline{X}'\ar[d]\\
X\ar[d]_{\pi}\ar@{-->}[r]&X'\ar[d]\\
Z&Z'
}
$$
We take a log resolution $\tilde{f}\colon\tilde{X}\to \overline{X}$ of $(\overline{X},{\rm Supp}\overline{B})$ such that $\boldsymbol{\rm M}=\overline{\boldsymbol{\rm M}_{\tilde{X}}}$ and the induced birational map $\tilde{X}\dashrightarrow \overline{X}'$ is a morphism. 
We set $\tilde{A}=\tilde{f}^{*}\overline{A}$. 
We can write 
\begin{equation*}\tag{2}\label{proof-thm--abund-sub-(2)}
K_{\tilde{X}}+\tilde{B}+\tilde{A}=\tilde{f}^{*}(K_{\overline{X}}+\overline{B}+\overline{A})+\tilde{E}
\end{equation*}
with $\tilde{B}\geq0$ and $\tilde{E}\geq 0$ having no common components. 
Replacing $A$, we may assume $\tilde{A}=\tilde{f}_{*}^{-1}\overline{A}$ and that $(\tilde{X},\tilde{B}+\tilde{A})$ is a log smooth lc pair. 
We may apply Lemma \ref{lem--lift} to $\bigl(\overline{X},\overline{B}+\overline{A},\boldsymbol{\rm M}\bigr)$, $\overline{X}\dashrightarrow \overline{X}'$, $\tilde{f}\colon\tilde{X}\to \overline{X}$, and $\bigl(\tilde{X},\tilde{B}+\tilde{A},\boldsymbol{\rm M}\bigr)$. 
Indeed, the first condition of Lemma \ref{lem--lift} follows from the above (\ref{proof-thm--abund-sub-(i)}), the second condition of Lemma \ref{lem--lift} is nothing but  (\ref{proof-thm--abund-sub-(ii)}) stated above, and the third condition of Lemma \ref{lem--lift} follows from the relation (\ref{proof-thm--abund-sub-(2)}). 
By applying Lemma \ref{lem--lift}, we obtain a birational contraction $\tilde{X}\dashrightarrow \tilde{X}'$ over $\overline{X}'$ which is a sequence of steps of a $(K_{\tilde{X}}+\tilde{B}+\tilde{A}+\boldsymbol{\rm M}_{\tilde{X}})$-MMP and a projective $\mathbb{Q}$-factorial generalized dlt pair $\bigl(\tilde{X}',\tilde{B}'+\tilde{A}',\boldsymbol{\rm M}\bigr)$ such that $K_{\tilde{X}'}+\tilde{B}'+\tilde{A}'+\boldsymbol{\rm M}_{\tilde{X}'}$ is equal to the pullback of $K_{\overline{X}'}+\overline{B}'+\overline{A}'+\boldsymbol{\rm M}_{\overline{X}'}$. 
By the property (\ref{proof-thm--abund-sub-(i)}), the divisor $K_{\tilde{X}'}+\tilde{B}'+\tilde{A}'+\boldsymbol{\rm M}_{\tilde{X}'}$ is semi-ample over $Z'$. 
Let $\tilde{X}'\to \tilde{Z}$ be a contraction over $Z'$ induced by $K_{\tilde{X}'}+\tilde{B}'+\tilde{A}'+\boldsymbol{\rm M}_{\tilde{X}'}$. 
It follows that the induced morphism $\tilde{Z}\to Z'$ is birational because the birational transform of $K_{\tilde{X}'}+\tilde{B}'+\tilde{A}'+\boldsymbol{\rm M}_{\tilde{X}'}$ on $X'$ is $K_{X'}+B'+A'+\boldsymbol{\rm M}_{X'}\sim_{\mathbb{R},Z}0$ which shows that the restriction of $K_{\tilde{X}'}+\tilde{B}'+\tilde{A}'+\boldsymbol{\rm M}_{\tilde{X}'}$ to a general fiber of $\tilde{X}'\to Z'$ is numerically trivial.
\end{step4}

\begin{step4}\label{step4abund}
We put $f=\overline{f}\circ\tilde{f}\colon \tilde{X}\to X$. 
We have constructed the following diagram
$$
\xymatrix@R=15pt{
\tilde{X}\ar[d]_{f}\ar@{-->}[r]&\tilde{X}'\ar[d]\\
X\ar[d]_{\pi}&\tilde{Z}\\
Z
}
$$
and projective $\mathbb{Q}$-factorial generalized dlt pairs $\bigl(\tilde{X},\tilde{B}+\tilde{A},\boldsymbol{\rm M}\bigr)$ and $\bigl(\tilde{X}',\tilde{B}'+\tilde{A}',\boldsymbol{\rm M}\bigr)$ 
such that $(\tilde{X},\tilde{B}+\tilde{A})$ is a log smooth lc pair and the relation $K_{\tilde{X}'}+\tilde{B}'+\tilde{A}'+\boldsymbol{\rm M}_{\tilde{X}'}\sim_{\mathbb{R},\tilde{Z}}0$ holds. 
We have also proved that the property of being abundant for $K_{X}+B+A+\boldsymbol{\rm M}_{X}$ is reduced to that for $K_{\tilde{X}'}+\tilde{B}'+\tilde{A}'+\boldsymbol{\rm M}_{\tilde{X}'}$. 
In this step, using divisorial adjunction for generalized pairs, we will construct new generalized lc pairs whose dimensions are ${\rm dim}X-1$.  

By the relations (\ref{proof-thm--abund-sub-(1)}) and (\ref{proof-thm--abund-sub-(2)}) in Step \ref{step3abund} and Lemma \ref{lem--logsmooth}, we can write 
\begin{equation*}\tag{3}\label{proof-thm--abund-sub-(3)}
K_{\tilde{X}}+\tilde{B}+\tilde{A}+\boldsymbol{\rm M}_{\tilde{X}}=f^{*}(K_{X}+B+A+\boldsymbol{\rm M}_{X})+\tilde{f}^{*}\overline{E}+\tilde{E}
\end{equation*}
and $\tilde{B}+\tilde{A}$ and $\tilde{f}^{*}\overline{E}+\tilde{E}$ have no common components. 
By (\ref{proof-thm--abund-sub-(3)}) and that $\tilde{X}\dashrightarrow \tilde{X}'$ is a sequence of steps of a $(K_{\tilde{X}}+\tilde{B}+\tilde{A}+\boldsymbol{\rm M}_{\tilde{X}})$-MMP, to prove Theorem \ref{thm--abund-sub}, it is sufficient to prove that $K_{\tilde{X}'}+\tilde{B}'+\tilde{A}'+\boldsymbol{\rm M}_{\tilde{X}}$ is abundant. 
We recall that $S$ is a component of $\llcorner B\lrcorner$ and the birational transform $S'$ on $X'$ is ample over $Z'$. 
Let $\tilde{S}$ be the birational transform of $S$ on $\tilde{X}$. 
Then $\tilde{S}$ is not contracted by the map $\tilde{X}\dashrightarrow \tilde{X}'$ because $K_{\tilde{X}}+\tilde{B}-\epsilon\tilde{S}+\tilde{A}+\boldsymbol{\rm M}_{\tilde{X}}$ is not pseudo-effective for any $\epsilon>0$ (this follows from the second sentence in Step \ref{step2abund}). 
Furthermore, the bitarional transform $\tilde{S}'$ of $S$ on $\tilde{X'}$ dominates $\tilde{Z}$ because $S'$ is ample over $Z'$ and $\tilde{Z}\to Z'$ is birational.  
Then the natural morphism $\tilde{S}'\to \tilde{Z}$ is surjective and $K_{\tilde{X}'}+\tilde{B}'+\tilde{A}'+\boldsymbol{\rm M}_{\tilde{X}'}\sim_{\mathbb{R},\tilde{Z}}0$. 
Thus, it is sufficient to prove that $(K_{\tilde{X}'}+\tilde{B}'+\tilde{A}'+\boldsymbol{\rm M}_{\tilde{X}'})|_{\tilde{S}'}$ is abundant (\cite[Remark 2..8 (2)]{hashizumehu}). 

By taking a suitable resolution of the normalization of the graph of $\tilde{X}\dashrightarrow \tilde{X}'$, we can find a common resolution $Y \to \tilde{X}$ and $Y\to \tilde{X}'$ and a subvariety $T\subset Y$ which is birational to $\tilde{S}$ and $\tilde{S}'$ such that the induced morphisms $T\to\tilde{S}$ and $T\to \tilde{S}'$ form a common resolution of $\tilde{S} \dashrightarrow \tilde{S}'$. 
Since $T$ has codimension one in $Y$, replacing $Y$ by a higher resolution if necessary, we may assume $\boldsymbol{\rm M}=\overline{\boldsymbol{\rm M}_{Y}}$.  
Pick $M_{T}\sim_{\mathbb{R}}\boldsymbol{\rm M}_{Y}|_{T}$ which is an $\mathbb{R}_{>0}$-linear combination of nef Cartier divisors on $T$, and denote the b-divisor $\overline{M_{T}}$ by $\boldsymbol{\rm N}$. 
With this $\boldsymbol{\rm N}$, by applying divisorial adjunction to $(X,B,\boldsymbol{\rm M})$, $\bigl(\tilde{X},\tilde{B},\boldsymbol{\rm M}\bigr)$, and $\bigl(\tilde{X}',\tilde{B}',\boldsymbol{\rm M}\bigr)$, we may construct generalized pairs
\begin{equation*}
\begin{split}
(S,B_{S},\boldsymbol{\rm N}),\quad 
(\tilde{S},B_{\tilde{S}},\boldsymbol{\rm N}), \quad {\rm and}\quad  
(\tilde{S}',B_{\tilde{S}'},\boldsymbol{\rm N}).
\end{split}
\end{equation*}
Here, we regard $\boldsymbol{\rm N}=\overline{M_{T}}$ as b-divisors on $S$, $\tilde{S}$, and $\tilde{S}'$. 
We take a common resolution $\tau\colon \tilde{T}\to \tilde{S}$ and $\tau'\colon \tilde{T}\to \tilde{S}'$ of the birational map $\tilde{S}\dashrightarrow \tilde{S}'$ such that $\tau$ and $\tau'$ are log resolutions of $(\tilde{S}, {\rm Supp}B_{\tilde{S}})$ and $(\tilde{S}', {\rm Supp}B_{\tilde{S}'})$ respectively, $\tau'^{-1}_{*}B_{\tilde{S}'}\cup{\rm Ex}(\tau)\cup{\rm Ex}(\tau')$ is a simple normal crossing divisor, and the induced  map $\tilde{T}\dashrightarrow T$ is a morphism. 
We can regarded $\boldsymbol{\rm N}$ as a b-divisor on $\tilde{T}$. 
We set 
$$A_{S}=A|_{S}, \quad A_{\tilde{S}}=\tilde{A}|_{\tilde{S}},\quad {\rm and} \quad A_{\tilde{S}'}=\tilde{A}'|_{\tilde{S}'}.$$
By construction of divisorial adjunction for generalized pairs \cite[Definition 4.7]{bz}, the generalized pair $(S,B_{S}+A_{S},\boldsymbol{\rm N})$ is equal to the generalized pair constructed with the divisorial adjunction for $(X,B+A,\boldsymbol{\rm M})$. 
Therefore, $(S,B_{S}+A_{S},\boldsymbol{\rm N})$ is generalized lc. 
By the same reasons, $(\tilde{S},B_{\tilde{S}}+A_{\tilde{S}},\boldsymbol{\rm N})$ and $(\tilde{S}',B_{\tilde{S}'}+A_{\tilde{S}'},\boldsymbol{\rm N})$ are generalized lc. 
Since 
$$K_{\tilde{S}'}+B_{\tilde{S}'}+A_{\tilde{S}'}+\boldsymbol{\rm N}_{\tilde{S}'}\sim_{\mathbb{R}}(K_{\tilde{X}'}+\tilde{B}'+\tilde{A}'+\boldsymbol{\rm M}_{\tilde{X}'})|_{\tilde{S}'},$$
to prove Theorem \ref{thm--abund-sub}, it is sufficient to prove that 
$K_{\tilde{S}'}+B_{\tilde{S}'}+A_{\tilde{S}'}+\boldsymbol{\rm N}_{\tilde{S}'}$ is abundant. 
Since $A_{\tilde{S}}=f^{*}A|_{\tilde{S}}$, by \cite[Lemma 2.1]{hashizumehu} and \cite[Lemma 2.2]{hashizumehu} and replacing $A$ with a general one, we may assume $A_{\tilde{S}}\geq0$, $A_{\tilde{S}'}\geq0$, $\tau^{*}A_{\tilde{S}}\leq \tau'^{-1}_{*}A_{\tilde{S}'}$, and $\tau'$ is a log resolution of $\bigl(\tilde{S}',{\rm Supp}(B_{\tilde{S}'}+A_{\tilde{S}'})\bigr)$. 
\end{step4}

\begin{step4}\label{step5abund}
Let $\pi_{S}\colon S\to Z$ be the restriction of $\pi\colon X\to Z$ to $S$, and let $f_{\tilde{S}}\colon \tilde{S}\to S$ be the restriction of $f\colon \tilde{X}\to X$ to $\tilde{S}$. 
We have the following diagram
\begin{equation*}
\xymatrix@R=13pt{
&
\tilde{T}\ar[dl]_{\tau}\ar[dr]^{\tau'}&\\
\tilde{S}\ar[d]_{f_{\tilde{S}}}\ar@{-->}[rr]&&\tilde{S}'\ar[d]\\
S\ar[d]_{\pi_{S}}&&\tilde{Z} \\
Z&&
}
\end{equation*}
and generalized lc pairs $(S,B_{S}+A_{S},\boldsymbol{\rm N})$, $(\tilde{S},B_{\tilde{S}}+A_{\tilde{S}},\boldsymbol{\rm N})$, and $(\tilde{S}',B_{\tilde{S}'}+A_{\tilde{S}'},\boldsymbol{\rm N})$. 
By Step \ref{step4abund}, the property of being abundant for $K_{\tilde{X}'}+\tilde{B}'+\tilde{A}'+\boldsymbol{\rm M}_{\tilde{X}'}$ was reduced to that for $K_{\tilde{S}'}+B_{\tilde{S}'}+A_{\tilde{S}'}+\boldsymbol{\rm N}_{\tilde{S}'}$. 
In this step, we will study the morphism $(S,B_{S}+A_{S},\boldsymbol{\rm N}) \to Z$, and we will construct a generalized lc pair on $\tilde{T}$ by using $(\tilde{S}',B_{\tilde{S}'}+A_{\tilde{S}'},\boldsymbol{\rm N})$. 

We check that $\pi_{S}\colon S\to Z$, $(S,B_{S},\boldsymbol{\rm N})$, $C|_{S}$, and $A_{S}\sim_{\mathbb{R}}\pi_{S}^{*}A_{Z}$ satisfy conditions of Theorem \ref{thm--abund-sub}. 
By construction, $(S,B_{S},\boldsymbol{\rm N})$ is a generalized lc pair such that  $\boldsymbol{\rm N}=\overline{M_{T}}$ is a finite $\mathbb{R}_{>0}$-linear combination of b-nef $\mathbb{Q}$-b-Cartier $\mathbb{Q}$-b-divisors. 
Thus, condition (\ref{thm--abund-sub-(I)}) of Theorem \ref{thm--abund-sub} is satisfied. 
Since $(X,B+tC, \boldsymbol{\rm M})$ is a generalized lc pair for some $t>0$ and $S$ is a component of $\llcorner B \lrcorner$, the divisor $C|_{S}$ is a well-defined effective $\mathbb{R}$-Cartier divisor on $S$. 
By construction of divisorial adjunction for generalized pairs \cite[Definition 4.7]{bz}, 
the generalized pair $(S,B_{S}+tC|_{S},\boldsymbol{\rm N})$ is equal to the generalized pair constructed with the divisorial adjunction for $(X,B+tC, \boldsymbol{\rm M})$. 
Since $(X,B+tC, \boldsymbol{\rm M})$ is generalized lc, we see that $(S,B_{S}+tC|_{S},\boldsymbol{\rm N})$ is generalized lc. 
Thus, condition (\ref{thm--abund-sub-(II-a)}) of Theorem \ref{thm--abund-sub} is satisfied. 
By construction, we have 
$$K_{S}+B_{S}+C|_{S}+\boldsymbol{\rm N}_{S}\sim_{\mathbb{R}}(K_{X}+B+C+\boldsymbol{\rm M}_{X})|_{S}\sim_{\mathbb{R},Z}0.$$ 
Thus, condition (\ref{thm--abund-sub-(II-b)}) of Theorem \ref{thm--abund-sub} is satisfied. 
Finally, $(S,B_{S}+A_{S},\boldsymbol{\rm N})$ is generalized lc. 
From these discussions, we see that $\pi_{S}\colon S\to Z$, $(S,B_{S},\boldsymbol{\rm N})$, $C|_{S}$, and $A_{S}\sim_{\mathbb{R}}\pi_{S}^{*}A_{Z}$ satisfy conditions of Theorem \ref{thm--abund-sub}. 
By induction hypothesis of Theorem \ref{thm--abund-sub}, the divisor $K_{S}+B_{S}+A_{S}+\boldsymbol{\rm N}_{S}$ is abundant. 

We put $A_{\tilde{T}}=\tau^{*}A_{\tilde{S}}$, which is equal to the pullback of $A_{S}=A|_{S}$ to $\tilde{T}$. 
Thanks to the relation $\tau^{*}A_{\tilde{S}}\leq \tau'^{-1}_{*}A_{\tilde{S}'}$, we can write
\begin{equation*}\tag{4}\label{proof-thm--abund-sub-(4)}
K_{\tilde{T}}+\Psi_{\tilde{T}}+A_{\tilde{T}}+\boldsymbol{\rm N}_{\tilde{T}}=\tau'^{*}(K_{\tilde{S}'}+B_{\tilde{S}'}+A_{\tilde{S}'}+\boldsymbol{\rm N}_{\tilde{S}'})+E_{\tilde{T}}
\end{equation*} 
where $\Psi_{\tilde{T}}+A_{\tilde{T}}$ and $E_{\tilde{T}}$ have no common components. 
Note that $(\tilde{T},\Psi_{\tilde{T}}+A_{\tilde{T}})$ is a log smooth lc pair because $\tau' \colon \tilde{T} \to \tilde{S}'$ is a log resolution of $\bigl(\tilde{S}',{\rm Supp}(B_{\tilde{S}'}+A_{\tilde{S}'})\bigr)$. 
By (\ref{proof-thm--abund-sub-(4)}) above, it is sufficient to prove that $K_{\tilde{T}}+\Psi_{\tilde{T}}+A_{\tilde{T}}+\boldsymbol{\rm N}_{\tilde{T}}$ is abundant. 
\end{step4}

\begin{step4}\label{step6abund} 
In this step, we prove that $K_{\tilde{T}}+\Psi_{\tilde{T}}+A_{\tilde{T}}+\boldsymbol{\rm N}_{\tilde{T}}$ is abundant. 
We may write
$$K_{\tilde{T}}+\Psi_{\tilde{T}}+A_{\tilde{T}}+\boldsymbol{\rm N}_{\tilde{T}}=(f_{\tilde{S}}\circ \tau)^{*}(K_{S}+B_{S}+A_{S}+\boldsymbol{\rm N}_{S})+\Xi_{+}-\Xi_{-},$$
where $\Xi_{+}\geq0$ and $\Xi_{-}\geq0$ have no common components. 
Since $A_{\tilde{T}}=(f_{\tilde{S}}\circ \tau)^{*}A_{S}$, to apply Lemma \ref{lem--abund-birat}, we show that $\Xi_{+}$ is exceptional over $S$. 

Suppose by contradiction that there is a component $Q$ of $\Xi_{+}$ which is not exceptional over $S$. 
Then 
$$a(Q,\tilde{T}, \Psi_{\tilde{T}}+A_{\tilde{T}}+\boldsymbol{\rm N}_{\tilde{T}})< a(Q,S,B_{S}+A_{S}+\boldsymbol{\rm N}_{S})\leq 1.$$
By (\ref{proof-thm--abund-sub-(3)}) in Step \ref{step4abund} and Lemma \ref{lem--adjunction}, we have
\begin{equation*}
a(Q,S,B_{S}+A_{S}+\boldsymbol{\rm N}_{S})=a(Q,\tilde{S}, B_{\tilde{S}}+A_{\tilde{S}}+\boldsymbol{\rm N}_{\tilde{S}})
\end{equation*}
Since the birational map $\tilde{X}\dashrightarrow \tilde{X}'$ is a sequence of steps of a $(K_{\tilde{X}}+\tilde{B}+\tilde{A}+\boldsymbol{\rm M}_{\tilde{X}})$-MMP, by taking a suitable common resolution of $\tilde{X}\dashrightarrow \tilde{X}'$ and the negativity lemma (that is the generalized pair analogue of \cite[Lemma 4.2.10]{fujino-sp-ter}), we have 
$$a(Q,\tilde{S}, B_{\tilde{S}}+A_{\tilde{S}}+\boldsymbol{\rm N}_{\tilde{S}})\leq a(Q,\tilde{S}', B_{\tilde{S}'}+A_{\tilde{S}'}+\boldsymbol{\rm N}_{\tilde{S}'}).$$
Finally, since $\Psi_{\tilde{T}}+A_{\tilde{T}}$ and $E_{\tilde{T}}$ have no common components (see (\ref{proof-thm--abund-sub-(4)}) in Step \ref{step5abund}), we can write 
$$a(Q,\tilde{T}, \Psi_{\tilde{T}}+A_{\tilde{T}}+\boldsymbol{\rm N}_{\tilde{T}})={\rm min}\{a(Q,\tilde{S}', B_{\tilde{S}'}+A_{\tilde{S}'}+\boldsymbol{\rm N}_{\tilde{S}'}),1\}.$$ 
By combining these relations, we obtain
\begin{equation*}
\begin{split}
a(Q,S,B_{S}+A_{S}+\boldsymbol{\rm N}_{S})&={\rm min}\{a(Q,S,B_{S}+A_{S}+\boldsymbol{\rm N}_{S}),1\}\\
&={\rm min}\{a(Q,\tilde{S}, B_{\tilde{S}}+A_{\tilde{S}}+\boldsymbol{\rm N}_{\tilde{S}}),1\}\\
&\leq{\rm min}\{a(Q,\tilde{S}', B_{\tilde{S}'}+A_{\tilde{S}'}+\boldsymbol{\rm N}_{\tilde{S}'}),1\}\\
&=a(Q,\tilde{T}, \Psi_{\tilde{T}}+A_{\tilde{T}}+\boldsymbol{\rm N}_{\tilde{T}})\\
&<a(Q,S,B_{S}+A_{S}+\boldsymbol{\rm N}_{S}),
\end{split}
\end{equation*}
which is a contradiction. 
In this way, we see that $\Xi_{+}$ is exceptional over $S$. 
By applying Lemma \ref{lem--abund-birat} to $(S,B_{S},\boldsymbol{\rm N})\to Z$, $A_{S}\sim_{\mathbb{R}}\pi^{*}_{S}A_{Z}$, $\tilde{T}\to S$, and $\Psi_{\tilde{T}}$, 
we see that the divisor $K_{\tilde{T}}+\Psi_{\tilde{T}}+A_{\tilde{T}}+\boldsymbol{\rm N}_{\tilde{T}}$ is abundant. 
We finish this step. 
\end{step4}
By (\ref{proof-thm--abund-sub-(4)}) in Step \ref{step5abund}, the divisor $K_{\tilde{S}'}+B_{\tilde{S}'}+A_{\tilde{S}'}+\boldsymbol{\rm N}_{\tilde{S}'}$ is abundant. 
By the arguments in Step \ref{step4abund}, the divisor $K_{\tilde{X}'}+\tilde{B}'+\tilde{A}'+\boldsymbol{\rm M}_{\tilde{X}'}$ is abundant. 
By construction of $\tilde{X}\dashrightarrow \tilde{X}'$ in Step \ref{step3abund}, the divisor $K_{\tilde{X}}+\tilde{B}+\tilde{A}+\boldsymbol{\rm M}_{\tilde{X}}$ is abundant. 
Finally, using (\ref{proof-thm--abund-sub-(3)}) in Step \ref{step4abund}, we see that the divisor $K_{X}+B+A+\boldsymbol{\rm M}_{X}$ is abundant. 
We are done. 
\end{proof}

\section{Proof of main results}\label{sec4}

In this section, we prove the main result of this paper. 

\subsection{Non-vanishing theorem}
In this subsection, we firstly prove Theorem \ref{thm--intro-abund-main} 
, then we prove Theorem \ref{thm--abund-main-2}, and finally we prove Theorem \ref{thm--non-vanishing-lc-divisor-2} which is a non-$\mathbb{R}$-Cartier analogue of Theorem \ref{thm--abund-main-2}. 
Theorem \ref{thm--gen-non-vanishing-lp} is a special case of Theorem \ref{thm--non-vanishing-lc-divisor-2}. 


\begin{thm}\label{thm--abund-main}
Let $\pi\colon X\to Z$ be a morphism of normal projective varieties, and let $(X,B, \boldsymbol{\rm M})$ be a $\mathbb{Q}$-factorial generalized dlt pair. 
Suppose that  
\makeatletter 
\renewcommand{\p@enumii}{II-} 
\makeatother
\begin{enumerate}[(I)]
\item \label{thm--abund-main-(I)}
$\boldsymbol{\rm M}$ is a finite $\mathbb{R}_{>0}$-linear combination of b-nef $\mathbb{Q}$-b-Cartier $\mathbb{Q}$-b-divisors, and  
\item \label{thm--abund-main-(II)}
there is an effective $\mathbb{R}$-divisor $C$ on $X$ such that
\begin{enumerate}[({II-}a)]
\item \label{thm--abund-main-(II-a)}
the generalized pair $(X,B+tC, \boldsymbol{\rm M})$ is generalized lc for some $t>0$, and  
\item \label{thm--abund-main-(II-b)}
$K_{X}+B+C+\boldsymbol{\rm M}_{X}\sim_{\mathbb{R},Z}0$.
\end{enumerate}
\end{enumerate}
Let $A_{Z}$ be an ample $\mathbb{R}$-divisor on $Z$. 
Pick $A\sim_{\mathbb{R}}\pi^{*}A_{Z}$ so that $A\geq0$ and $(X,B+A, \boldsymbol{\rm M})$ is a generalized lc pair whose generalized lc centers are those of $(X,B, \boldsymbol{\rm M})$. 

We put $\Delta=B+A$. 
Then, for any sequence of steps of a $(K_{X}+\Delta+\boldsymbol{\rm M}_{X})$-MMP
$$(X,\Delta, \boldsymbol{\rm M})=:(X_{0},\Delta_{0},\boldsymbol{\rm M}) \dashrightarrow \cdots \dashrightarrow (X_{i},\Delta_{i},\boldsymbol{\rm M}) \dashrightarrow \cdots,$$
the divisor $K_{X_{i}}+\Delta_{i}+\boldsymbol{\rm M}_{X_{i}}$ is log abundant with respect to $(X_{i},\Delta_{i},\boldsymbol{\rm M})$ for every $i$.  
Furthermore, if $(X,B, \boldsymbol{\rm M})$ is generalized klt and $K_{X}+B+A+\boldsymbol{\rm M}_{X}$ is pseudo-effective, then $K_{X}+B+A+\boldsymbol{\rm M}_{X}$ birationally has the Nakayama--Zariski decomposition with semi-ample positive part. 
\end{thm}

\begin{proof}
When $(X,B, \boldsymbol{\rm M})$ is generalized klt and $K_{X}+B+A+\boldsymbol{\rm M}_{X}$ is pseudo-effective, the divisor $C$ of (\ref{thm--abund-main-(II)}) is vertical over $Z$. 
Thus, we can find an open subset $U \subset Z$ such that $(K_{X}+B+\boldsymbol{\rm M}_{X})|_{\pi^{-1}(U)}\sim_{\mathbb{R},U}0$. 
Then the second assertions directly follows from Proposition \ref{prop--genabklt}. 
So it is enough to show the property of being log abundant. 

The strategy is very similar to \cite[Proof of Theorem 5.5]{hashizumehu}. 
Let $(X,B, \boldsymbol{\rm M})\to Z$ and $A$ be as in Theorem \ref{thm--abund-main}. 
Put $\Delta=B+A$, and let
$$(X,\Delta, \boldsymbol{\rm M})=:(X_{0},\Delta_{0},\boldsymbol{\rm M}) \dashrightarrow \cdots \dashrightarrow (X_{i},\Delta_{i},\boldsymbol{\rm M}) \dashrightarrow \cdots$$
be a sequence of steps of a $(K_{X}+\Delta+\boldsymbol{\rm M}_{X})$-MMP. 
Fix an index $i$. 
By Theorem \ref{thm--abund-sub}, the divisor $K_{X}+B+A+\boldsymbol{\rm M}_{X}$ is abundant, so $K_{X_{i}}+\Delta_{i}+\boldsymbol{\rm M}_{X_{i}}$ is also abundant. 
Therefore, it is sufficient to prove that $(K_{X_{i}}+\Delta_{i}+\boldsymbol{\rm M}_{X_{i}})|_{S_{i}}$ is abundant for every generalized lc center $S_{i}$ of $(X_{i},\Delta_{i},\boldsymbol{\rm M})$. 
We fix a generalized lc center $S_{i}$ of $(X_{i},\Delta_{i},\boldsymbol{\rm M})$. 
Then there is a generalized lc center $S$ of $(X,\Delta, \boldsymbol{\rm M})$ such that the map $X\dashrightarrow X_{i}$ induces a birational map $S\dashrightarrow S_{i}$. 

Let $f\colon (X',\Delta',\boldsymbol{\rm M}) \to (X,B+A, \boldsymbol{\rm M})$ be a $\mathbb{Q}$-factorial generalized dlt model such that there is a component $S'$ of $\llcorner \Delta' \lrcorner$ from which $f$ induces a surjective morphism $S' \to S$. 
Such model always exists (Remark \ref{rem--gen-dlt} (\ref{rem--gen-dlt-(2)})). 
Taking a lift of the $(K_{X}+\Delta+\boldsymbol{\rm M}_{X})$-MMP as in \cite[3.5]{hanli}, we get a sequence of steps of a $(K_{X'}+\Delta'+\boldsymbol{\rm M}_{X'})$-MMP
$$(X',\Delta', \boldsymbol{\rm M})=:(X'_{0},\Delta'_{0},\boldsymbol{\rm M}) \dashrightarrow \cdots \dashrightarrow (X'_{k_{i}},\Delta'_{k_{i}},\boldsymbol{\rm M})$$
and $\mathbb{Q}$-factorial generalized dlt models $f_{i}\colon (X'_{k_{i}},\Delta'_{k_{i}},\boldsymbol{\rm M}) \to (X_{i},\Delta_{i}, \boldsymbol{\rm M})$ such that the birational transform $S'_{k_{i}}$ of $S'$ on $X'_{k_{i}}$ is a component of $\llcorner \Delta'_{k_{i}}\lrcorner$ and $f_{i}\colon X'_{k_{i}}\to X_{i}$ induces a surjective morphism $S'_{k_{i}}\to S_{i}$. 
By construction, we have
$$K_{X'_{k_{i}}}+\Delta'_{k_{i}}+\boldsymbol{\rm M}_{X'_{k_{i}}}=f_{i}^{*}(K_{X_{i}}+\Delta_{i}+\boldsymbol{\rm M}_{X_{i}}),$$
thus $(K_{X_{i}}+\Delta_{i}+\boldsymbol{\rm M}_{X_{i}})|_{S_{i}}$ is abundant if and only if $(K_{X'_{k_{i}}}+\Delta'_{k_{i}}+\boldsymbol{\rm M}_{X'_{k_{i}}})|_{S'_{k_{i}}}$ is abundant. 
Furthermore, it is easy to check that the conditions (\ref{thm--abund-main-(I)}), (\ref{thm--abund-main-(II-a)}), and (\ref{thm--abund-main-(II-b)}) in Theorem \ref{thm--abund-main} hold after replacing $(X,B+A, \boldsymbol{\rm M})$ by $(X',\Delta',\boldsymbol{\rm M})$. 
From these arguments, by replacing the $(K_{X}+\Delta+\boldsymbol{\rm M}_{X})$-MMP and $S_{i}$, we may assume that $S_{i}$ is a component of $\llcorner \Delta_{i} \lrcorner$. 
By the assumption of Theorem \ref{thm--abund-main} that $(X,B, \boldsymbol{\rm M})$ is a $\mathbb{Q}$-factorial generalized dlt pair and $(X,B+A, \boldsymbol{\rm M})$ is a generalized lc pair whose generalized lc centers are those of $(X,B, \boldsymbol{\rm M})$, it follows that 
$S$ is a component of $\llcorner B \lrcorner$. 
Let $B_{i}$ and $A_{i}$ be the birational transforms of $B$ and $A$ on $X_{i}$, respectively. 
Then $S_{i}$ is a component of $\llcorner B_{i} \lrcorner$. 

Taking a suitable resolution of the normalization of the graph of $X \dashrightarrow X_{i}$, we can find a common resolution $Y \to X$ and $Y\to X_{i}$ of $X \dashrightarrow X_{i}$ and a subvariety $T\subset Y$ which is birational to $S$ and $S_{i}$ such that the induced morphisms $T\to S$ and $T\to S_{i}$ form a common resolution of $S\dashrightarrow S_{i}$. 
Since $T$ has codimension one in $Y$, replacing $Y$ by a higher resolution if necessary, we may assume $\boldsymbol{\rm M}=\overline{\boldsymbol{\rm M}_{Y}}$.  
We pick $M_{T}\sim_{\mathbb{R}}\boldsymbol{\rm M}_{Y}|_{T}$ which is an $\mathbb{R}_{>0}$-linear combination of nef Cartier divisors on $T$, and we denote the b-divisor $\overline{M_{T}}$ by $\boldsymbol{\rm N}$. 
With this $\boldsymbol{\rm N}$, by applying divisorial adjunction to $(X,B,\boldsymbol{\rm M})$ and $(X_{i},B_{i}, \boldsymbol{\rm M})$, we construct generalized pairs
$$(S,B_{S},\boldsymbol{\rm N}) \quad {\rm and} \quad (S_{i},B_{S_{i}}, \boldsymbol{\rm N}).$$
Here, we regard $\boldsymbol{\rm N}$ as b-divisors on $S$ and $S_{i}$. 
We take a common resolution $\tau\colon \tilde{T}\to S$ and $\tau'\colon \tilde{T}\to S_{i}$ of the birational map $S\dashrightarrow S_{i}$ such that $\tau$ and $\tau'$ are log resolutions of $(S, {\rm Supp}B_{S})$ and $(S_{i}, {\rm Supp}B_{S_{i}})$ respectively, $\tau'^{-1}_{*}B_{S_{i}}\cup{\rm Ex}(\tau)\cup{\rm Ex}(\tau')$ is a simple normal crossing divisor, and the induced birational map $\tilde{T}\dashrightarrow T$ is a morphism. 
We can regard $\boldsymbol{\rm N}$ as a b-divisor on $\tilde{T}$. 
We set 
$$A_{S}=A|_{S}\quad  {\rm and} \quad A_{S_{i}}=A_{i}|_{S_{i}}.$$ 
By construction of divisorial adjunction for generalized pairs \cite[Definition 4.7]{bz}, the generalized pair $(S,B_{S}+A_{S},\boldsymbol{\rm N})$ is equal to the generalized pair constructed with the divisorial adjunction for $(X,B+A,\boldsymbol{\rm M})$. 
Therefore, $(S,B_{S}+A_{S},\boldsymbol{\rm N})$ is generalized lc. 
By the same reason, $(S_{i},B_{S_{i}}+A_{S_{i}},\boldsymbol{\rm N})$ is also generalized lc. 
By \cite[Lemma 2.1]{hashizumehu} and \cite[Lemma 2.2]{hashizumehu} and replacing $A$ with a general one, we may assume $A_{S}\geq0$, $A_{S_{i}}\geq0$, $\tau^{*}A_{S}\leq \tau'^{-1}_{*}A_{S_{i}}$, and $\tau'\colon \tilde{T}\to S_{i}$ is a log resolution of $\bigl(S_{i},{\rm Supp}(B_{S_{i}}+A_{S_{i}})\bigr)$. 
Since 
$$K_{S_{i}}+B_{S_{i}}+A_{S_{i}}+\boldsymbol{\rm N}_{S_{i}} \sim_{\mathbb{R}}(K_{X_{i}}+\Delta_{i}+\boldsymbol{\rm M}_{X_{i}})|_{S_{i}},$$
it is sufficient to prove that $K_{S_{i}}+B_{S_{i}}+A_{S_{i}}+\boldsymbol{\rm N}_{S_{i}}$ is abundant. 

Since $\tau^{*}A_{S}\leq \tau'^{-1}_{*}A_{S_{i}}$, we can write 
\begin{equation*}\tag{$\clubsuit$}\label{proof--thm--abund-main-(clubsuit)}
\begin{split}
K_{\tilde{T}}+\Psi_{\tilde{T}}+\tau^{*}A_{S}+\boldsymbol{\rm N}_{\tilde{T}}=\tau'^{*}(K_{S_{i}}+B_{S_{i}}+A_{S_{i}}+\boldsymbol{\rm N}_{S_{i}})+E_{\tilde{T}}
\end{split}
\end{equation*}
with $\Psi_{\tilde{T}}\geq 0$ and $E_{\tilde{T}}\geq0$  such that $\Psi_{\tilde{T}}+\tau^{*}A_{S}$ and $E_{\tilde{T}}$ have no common components. 
Since $\tau'\colon \tilde{T}\to S_{i}$ is a log resolution of $\bigl(S_{i},{\rm Supp}(B_{S_{i}}+A_{S_{i}})\bigr)$, the pair $(\tilde{T}, \Psi_{\tilde{T}}+\tau^{*}A_{S})$ is a log smooth lc pair. 
Now we can write
$$K_{\tilde{T}}+\Psi_{\tilde{T}}+\tau^{*}A_{S}+\boldsymbol{\rm N}_{\tilde{T}}=\tau^{*}(K_{S}+B_{S}+A_{S}+\boldsymbol{\rm N}_{S})+\Xi_{+}-\Xi_{-}$$
such that $\Xi_{+}\geq0$ and $\Xi_{-}\geq0$ have no common components. 
Pick a component $Q$ of $\Xi_{+}$. 
Since $X\dashrightarrow X_{i}$ is a sequence of steps of a $(K_{X}+B+A+\boldsymbol{\rm M}_{X})$-MMP, by taking a suitable common resolution of $X \dashrightarrow X_{i}$ and the negativity lemma (or applying the generalized pair analogue of \cite[Lemma 4.2.10]{fujino-sp-ter}), we have 
$$a(Q,S, B_{S}+A_{S}+\boldsymbol{\rm N}_{S})\leq a(Q,S_{i}, B_{S_{i}}+A_{S_{i}}+\boldsymbol{\rm N}_{S_{i}}).$$
By construction of $\Psi_{\tilde{T}}$ and $E_{\tilde{T}}$, we can write 
$$a(Q,\tilde{T}, \Psi_{\tilde{T}}+\tau^{*}A_{S}+\boldsymbol{\rm N}_{\tilde{T}})={\rm min}\{a(Q,S_{i}, B_{S_{i}}+A_{S_{i}}+\boldsymbol{\rm N}_{S_{i}}),1\}.$$ 
Thus, we obtain
\begin{equation*}
\begin{split}
a(Q,S, B_{S}+A_{S}+\boldsymbol{\rm N}_{S})&>a(Q,\tilde{T}, \Psi_{\tilde{T}}+\tau^{*}A_{S}+\boldsymbol{\rm N}_{\tilde{T}})\\
&={\rm min}\{a(Q,S_{i}, B_{S_{i}}+A_{S_{i}}+\boldsymbol{\rm N}_{S_{i}}),1\}\\
&\geq{\rm min}\{a(Q,S, B_{S}+A_{S}+\boldsymbol{\rm N}_{S}),1\}.
\end{split}
\end{equation*}
Thus, it follows that $a(Q,S, B_{S}+A_{S}+\boldsymbol{\rm N}_{S})>1$ which shows that $Q$ is $\tau$-exceptional. 
Therefore, $\Xi_{+}$ is $\tau$-exceptional. 

Now we can apply Lemma \ref{lem--abund-birat} to $(S,B_{S}, \boldsymbol{\rm N})\to Z$, $C|_{S}$, $A_{S}$, $\tau\colon \tilde{T}\to S$, and $\Psi_{\tilde{T}}$. 
Then we see that $K_{\tilde{T}}+\Psi_{\tilde{T}}+\tau^{*}A_{S}+\boldsymbol{\rm N}_{\tilde{T}}$ is abundant. 
By (\ref{proof--thm--abund-main-(clubsuit)}), the divisor $K_{S_{i}}+B_{S_{i}}+A_{S_{i}}+\boldsymbol{\rm N}_{S_{i}}$ is abundant, therefore $(K_{X_{i}}+\Delta_{i}+\boldsymbol{\rm M}_{X_{i}})|_{S_{i}}$ is abundant. 
Since $i$ is arbitrary index and $S_{i}$ is arbitrary generalized lc center of $(X_{i},\Delta_{i},\boldsymbol{\rm M})$, we see that the divisor $K_{X_{i}}+\Delta_{i}+\boldsymbol{\rm M}_{X_{i}}$ is log abundant with respect to $(X_{i},\Delta_{i},\boldsymbol{\rm M})$ for every $i$. 
So we are done. 
\end{proof}

\begin{proof}[Proof of Theorem \ref{thm--intro-abund-main}]
The theorem is the case of Theorem \ref{thm--abund-main} where $\pi$ is the morphism of a $\mathbb{Q}$-factorial dlt model and $C=0$. 
\end{proof}

\begin{proof}[Proof of Theorem \ref{thm--abund-main-2}]
The first statement immediately follows from Theorem \ref{thm--abund-sub}. 
Thus, it is enough to prove the second statement. 
Suppose that $\boldsymbol{\rm M}_{X}$ is $\mathbb{R}$-Cartier. 
Fix a real number $\alpha \geq 1$. 
We pick a real number $\epsilon\in (0,1)$ such that $\epsilon \boldsymbol{\rm M}_{X}+A$ is ample, and we pick general $A_{\epsilon}\sim_{\mathbb{R}} \epsilon \boldsymbol{\rm M}_{X}+A$. 
Then $K_{X}+B+A_{\epsilon}+ (1-\epsilon)\boldsymbol{\rm M}_{X}$ is pseudo-effective. 
Let $f \colon \tilde{X} \to X$ be a log resolution of $(X,{\rm Supp}B)$ such that $\boldsymbol{\rm M}=\overline{\boldsymbol{\rm M}_{\tilde{X}}}$. 
We may write 
$$K_{\tilde{X}}+\tilde{B}+(1-\epsilon)\boldsymbol{\rm M}_{\tilde{X}}=f^{*}(K_{X}+B+(1-\epsilon)\boldsymbol{\rm M}_{X})+\tilde{E}$$
with effective $\mathbb{R}$-divisors $\tilde{B}$ and $\tilde{E}$ on $\tilde{X}$ having no common components. 
Furthermore, since $\boldsymbol{\rm M}_{X}$ is $\mathbb{R}$-Cartier, we may write 
$$\boldsymbol{\rm M}_{\tilde{X}}+\tilde{F}=f^{*}\boldsymbol{\rm M}_{X}$$
 for some effective $f$-exceptional $\mathbb{R}$-divisor $\tilde{F}$ on $\tilde{X}$. 
Since $(X,B, \boldsymbol{\rm M})$ is generalized lc, we can find a real number $\delta>0$ such that $(\tilde{X},\tilde{B}+\delta \tilde{F})$ is a log smooth lc pair. 
Put $\alpha'=\alpha-\epsilon$. 
Note that $\alpha'\geq 1-\epsilon$. 
Since $\alpha'=(\alpha-1)+(1-\epsilon)$, we have
$$K_{\tilde{X}}+\tilde{B}+f^{*}A_{\epsilon}+\alpha' \boldsymbol{\rm M}_{\tilde{X}}=f^{*}(K_{X}+B+A_{\epsilon}+\alpha' \boldsymbol{\rm M}_{X})+\tilde{E}- (\alpha-1) \tilde{F}.$$
Since $K_{X}+B+A_{\epsilon}+ (1-\epsilon)\boldsymbol{\rm M}_{X}$ is pseudo-effective, we see that $K_{\tilde{X}}+\tilde{B}+f^{*}A_{\epsilon}+(1-\epsilon) \boldsymbol{\rm M}_{\tilde{X}}$ is pseudo-effective. 
Since $\alpha'\geq 1-\epsilon$, we see that the divisor $K_{\tilde{X}}+\tilde{B}+f^{*}A_{\epsilon}+\alpha' \boldsymbol{\rm M}_{\tilde{X}}$ is pseudo-effective. 

We run a $(K_{\tilde{X}}+\tilde{B}+\alpha' \boldsymbol{\rm M}_{\tilde{X}})$-MMP over $X$ with scaling of an ample divisor. 
By the negativity lemma, after finitely many steps we obtain a projective birational morphism $f'\colon \tilde{X}' \to X$ and a projective $\mathbb{Q}$-factorial generalized dlt pair $(\tilde{X}',\tilde{B}', \alpha' \boldsymbol{\rm M})$ such that the effective part of $\tilde{E}- (\alpha-1) \tilde{F}$ is contracted by the birational map $\tilde{X}\dashrightarrow \tilde{X}'$. 
By replacing $\delta$ with a smaller one, we may assume that $\tilde{X}\dashrightarrow \tilde{X}'$ is a sequence of steps of a $(K_{\tilde{X}}+\tilde{B}+\delta \tilde{F}+\alpha' \boldsymbol{\rm M}_{\tilde{X}})$-MMP. 
Let $\tilde{E}'$ and $\tilde{F}'$ be the birational transforms of $\tilde{E}$ and $\tilde{F}$ on $\tilde{X}'$, respectively. 
By construction, the generalized pair $(\tilde{X}',\tilde{B}'+\delta \tilde{F}', \alpha' \boldsymbol{\rm M})$ is generalized lc, $(\alpha-1)\tilde{F}'-\tilde{E}'$ is effective, and
$$K_{\tilde{X}'}+\tilde{B}'+\alpha' \boldsymbol{\rm M}_{\tilde{X}'}+\bigl((\alpha-1)\tilde{F}'-\tilde{E}'\bigr)=f'^{*}(K_{X}+B+\alpha' \boldsymbol{\rm M}_{X}).$$
Now the $\mathbb{Q}$-factorial generalized dlt pair $(\tilde{X}',\tilde{B}', \alpha' \boldsymbol{\rm M})$ and the morphism $f'\colon \tilde{X}'\to X$ satisfy the conditions of Theorem \ref{thm--abund-sub}. 
Indeed, condition (\ref{thm--abund-sub-(I)}) of Theorem \ref{thm--abund-sub} is obvious, 
the effective $\mathbb{R}$-divisor $(\alpha-1)\tilde{F}'-\tilde{E}'$ satisfies  (\ref{thm--abund-sub-(II-b)}) of Theorem \ref{thm--abund-sub} because of the above relation, and $(\alpha-1)\tilde{F}'-\tilde{E}'$ satisfies (\ref{thm--abund-sub-(II-a)}) of Theorem \ref{thm--abund-sub} because $(\tilde{X}',\tilde{B}'+\delta \tilde{F}', \alpha' \boldsymbol{\rm M})$ is a generalized lc pair and ${\rm Supp}\bigl((\alpha-1)\tilde{F}'-\tilde{E}'\bigr)\subset {\rm Supp}\tilde{F}'$ which show the existence of a real number $t>0$ such that the generalized pair $\bigl(\tilde{X}',\tilde{B}'+t((\alpha-1)\tilde{F}'-\tilde{E}'), \alpha' \boldsymbol{\rm M}\bigr)$ is generalized lc. 
Applying Theorem \ref{thm--abund-sub} to $(\tilde{X}',\tilde{B}', \alpha' \boldsymbol{\rm M})\to X$ and $A_{\epsilon}$, we see that the divisor $K_{\tilde{X}'}+\tilde{B}'+\alpha' \boldsymbol{\rm M}_{\tilde{X}'}+f'^{*}A_{\epsilon}$ is in particular abundant. 
Since $K_{\tilde{X}}+\tilde{B}+f^{*}A_{\epsilon}+\alpha'\boldsymbol{\rm M}_{\tilde{X}}$ is pseudo-effective, $K_{\tilde{X}'}+\tilde{B}'+f'^{*}A_{\epsilon}+\alpha'\boldsymbol{\rm M}_{\tilde{X}'}$ is also pseudo-effective. 
Thus, we can find an effective $\mathbb{R}$-divisor $\tilde{D}'_{\alpha}$ on $\tilde{X}'$ such that 
$$\tilde{D}'_{\alpha}\sim_{\mathbb{R}}K_{\tilde{X}'}+\tilde{B}'+f'^{*}A_{\epsilon}+\alpha' \boldsymbol{\rm M}_{\tilde{X}'}.$$
By putting $D_{\alpha}=f'_{*}\tilde{D}'_{\alpha}$ and using $\alpha'=\alpha-\epsilon$ and $A_{\epsilon}\sim_{\mathbb{R}} \epsilon \boldsymbol{\rm M}_{X}+A$, 
 we obtain 
$$D_{\alpha}\sim_{\mathbb{R}}K_{X}+B+A_{\epsilon}+\alpha' \boldsymbol{\rm M}_{X}\sim_{\mathbb{R}} K_{X}+B+A+\alpha \boldsymbol{\rm M}_{X}.$$ 
Thus, the second statement of Theorem \ref{thm--abund-main-2} holds. 
We are done. 
\end{proof}

From now on, we discuss a non-$\mathbb{R}$-Cartier analogue of Theorem \ref{thm--abund-main-2}.

\begin{defn}[Pseudo-effective $\mathbb{R}$-divisor]\label{defn--positivity-divisor}
Let $X$ be a normal projective variety, and let $D$ be an $\mathbb{R}$-divisor on $X$, which is not necessarily $\mathbb{R}$-Cartier. 
We say that $D$ is {\em pseudo-effective} if $D+A$ is $\mathbb{R}$-linearly equivalent to an effective $\mathbb{R}$-divisor for all ample $\mathbb{R}$-divisors $A$ on $X$. 
\end{defn}

\begin{lem}\label{lem--equiv-positivity}
Let $X$ be a normal projective variety, and let $D$ be an $\mathbb{R}$-divisor on $X$. 
Let $f \colon Y \to X$ be a small birational morphism from a normal projective variety $Y$ such that $f^{-1}_{*}D$ is $\mathbb{R}$-Cartier. 
Then $D$ is pseudo-effective if and only if $f^{-1}_{*}D$ is pseudo-effective in the usual sense. 
\end{lem}

\begin{proof}
Since $f^{-1}_{*}D$ is $\mathbb{R}$-Cartier, $f^{-1}_{*}D$ is pseudo-effective if and only if $f^{-1}_{*}D+f^{*}A$ is $\mathbb{R}$-linearly equivalent to an effective $\mathbb{R}$-divisor for all ample $\mathbb{R}$-divisors $A$ on $X$. 
Since $f$ is small, for all ample $A$, the existence of $E_{Y}\geq 0$ satisfying $E_{Y}\sim_{\mathbb{R}}f^{-1}_{*}D+f^{*}A$ is equivalent to the existence of $E\geq 0$ satisfying $E\sim_{\mathbb{R}}D+A$. 
By Definition \ref{defn--positivity-divisor}, we see that $D$ is pseudo-effective if and only if $f^{-1}_{*}D$ is pseudo-effective. 
\end{proof}

\begin{thm}\label{thm--non-vanishing-lc-divisor-2}
Let $(X,B,\boldsymbol{\rm M})$ be a projective generalized lc pair such that $\boldsymbol{\rm M}$ is a finite $\mathbb{R}_{>0}$-linear combination of b-nef $\mathbb{Q}$-b-Cartier $\mathbb{Q}$-b-divisors. 
Let $A$ be an ample $\mathbb{R}$-divisor on $X$. 
Suppose that there is $t\in[0,1)$ such that $K_{X}+B+A+t\boldsymbol{\rm M}_{X}$ is pseudo-effective in the sense of Definition \ref{defn--positivity-divisor}. 

Then, for every real number $\alpha \geq t$, there exists an effective $\mathbb{R}$-divisor $D_{\alpha}$ on $X$ such that $K_{X}+B+A+\alpha \boldsymbol{\rm M}_{X} \sim_{\mathbb{R}}D_{\alpha}$.
\end{thm}

\begin{proof}
Fix a real number $\alpha \geq t$.
By \cite[Proposition 4.12]{has-class}, with notations as in \cite{has-class}, the pair of $X$ and $B$ is pseudo-lc defined in \cite[Definition 4.2]{has-class}. 
By \cite[Theorem 1.2]{has-class}, there exists a small birational morphism $h \colon Y\to X$ from a normal projective variety $Y$ such that if we put $B_{Y}=h^{-1}_{*}B$, then $K_{Y}+B_{Y}$ is an $h$-ample $\mathbb{R}$-Cartier divisor on $Y$. 
We can find a real number $\epsilon>0$ such that $\tfrac{t}{1-\epsilon}<1$ and  $\epsilon (K_{Y}+B_{Y})+h^{*}A_{Y}$ is an ample $\mathbb{R}$-divisor on $Y$. 
Putting $A_{Y}=\tfrac{1}{1-\epsilon}(\epsilon(K_{Y}+B_{Y})+h^{*}A_{Y})$, we have
$$K_{Y}+B_{Y}+h^{*}A+t\boldsymbol{\rm M}_{Y}= (1-\epsilon)\left(K_{Y}+B_{Y}+A_{Y}+\frac{t}{1-\epsilon}\boldsymbol{\rm M}_{Y}\right).$$
Then $\boldsymbol{\rm M}_{Y}$ is $\mathbb{R}$-Cartier and pseudo-effective. 
Since $K_{X}+B+A+t\boldsymbol{\rm M}_{X}$ is pseudo-effective by hypothesis of Theorem \ref{thm--non-vanishing-lc-divisor-2}, by Lemma \ref{lem--equiv-positivity}, we see that $K_{Y}+B_{Y}+A_{Y}+\tfrac{t}{1-\epsilon}\boldsymbol{\rm M}_{Y}$ is pseudo-effective with $\frac{t}{1-\epsilon}<1$. 
Therefore, for any $\alpha \geq t$, the relation
$$K_{Y}+B_{Y}+h^{*}A+\alpha\boldsymbol{\rm M}_{Y}= (1-\epsilon)\left(K_{Y}+B_{Y}+A_{Y}+\frac{\alpha}{1-\epsilon}\boldsymbol{\rm M}_{Y}\right)$$
holds and $K_{Y}+B_{Y}+A_{Y}+\frac{\alpha}{1-\epsilon}\boldsymbol{\rm M}_{Y}$ is pseudo-effective. 
Since $h \colon Y \to X$ is small, the generalized pair $(Y,B_{Y},\boldsymbol{\rm M})$ is generalized lc. 
If there is an effective $\mathbb{R}$-divisor $D_{Y}$ on $Y$ such that $D_{Y}\sim_{\mathbb{R}} K_{Y}+B_{Y}+A_{Y}+\frac{\alpha}{1-\epsilon}\boldsymbol{\rm M}_{Y}$, then $(1-\epsilon)h_{*}D_{Y}$ is effective and it satisfies $(1-\epsilon)h_{*}D_{Y}\sim_{\mathbb{R}}K_{X}+B+A+\alpha \boldsymbol{\rm M}_{X}$. 
Thus, it is sufficient to show the existence of an effective $\mathbb{R}$-divisor $D_{Y}$ on $Y$ such that $D_{Y}\sim_{\mathbb{R}} K_{Y}+B_{Y}+A_{Y}+\frac{\alpha}{1-\epsilon}\boldsymbol{\rm M}_{Y}$. 
In this way, by replacing $X$ (resp.~$B$, $A$, $t$, $\alpha$) with $Y$ (resp.~$B_{Y}$, $A_{Y}$, $\frac{t}{1-\epsilon}$, $\frac{\alpha}{1-\epsilon}$), we may assume that $\boldsymbol{\rm M}_{X}$ is $\mathbb{R}$-Cartier. 
Note that this condition shows that the pair $(X,B)$ is lc. 

If $\alpha=0$, then Theorem \ref{thm--non-vanishing-lc-divisor-2} follows from \cite[Theorem 1.5]{hashizumehu}. 
Therefore, we may assume $\alpha>0$. 
Put $t'={\rm min}\{1,\alpha\}$. 
Then, the generalized pair $(X,B,t'\boldsymbol{\rm M})$ is generalized lc. 
By hypothesis, $K_{X}+B+A+t'\boldsymbol{\rm M}_{X}$ is pseudo-effective. 
By Theorem \ref{thm--abund-main-2} (\ref{thm--abund-main-2-(2)}), we can find an effective $\mathbb{R}$-divisor $D_{\alpha}$ on $X$ such that $K_{X}+B+A+\alpha \boldsymbol{\rm M}_{X} \sim_{\mathbb{R}}D_{\alpha}$.
\end{proof}

\begin{proof}[Proof of Theorem \ref{thm--gen-non-vanishing-lp}]
It is the case of $t=0$ of Theorem \ref{thm--non-vanishing-lc-divisor-2}. 
\end{proof}

The following result is out of interests of this paper, but we write down for possibility of future use.

\begin{thm}\label{thm--non-vanishing-lc-divisor}
Let $(X,B,\boldsymbol{\rm M})$ be a projective generalized lc pair, and let $A$ be an ample $\mathbb{R}$-divisor on $X$. 
Then the followings hold.
\begin{enumerate}[(1)]
\item \label{thm--non-vanishing-lc-divisor-(1)}
If $K_{X}+B+A$ is pseudo-effective in the sense of Definition \ref{defn--positivity-divisor}, then there exists an effective $\mathbb{R}$-divisor $D$ on $X$ such that $K_{X}+B+A \sim_{\mathbb{R}}D$.  
\item \label{thm--non-vanishing-lc-divisor-(2)}
Let $G$ be a $\mathbb{Q}$-divisor on $X$ such that $G \sim_{\mathbb{R}} K_{X}+B+A$. 
Then the graded ring 
$$\mathcal{R}(X,G)=\underset{m \in \mathbb{Z}_{\geq 0}}{\bigoplus}H^{0}(X,\mathcal{O}_{X}(\llcorner mG \lrcorner))$$
is a finitely generated.
\end{enumerate}
\end{thm}

\begin{proof}
The first assertion is the case of $\alpha=0$ of Theorem \ref{thm--non-vanishing-lc-divisor-2}. 
We will prove the second assertion. 
By \cite[Proposition 4.12]{has-class}, notations as in \cite{has-class}, the pair $\langle X,B \rangle$ of $X$ and $B$ is pseudo-lc in the sense of \cite[Definition 4.2]{has-class}. 
By \cite[Theorem 1.2]{has-class}, there exists a small birational morphism $h \colon Y\to X$ from a normal projective variety $Y$ such that if we put $B_{Y}=h^{-1}_{*}B$, then $K_{Y}+B_{Y}$ is an $h$-ample $\mathbb{R}$-Cartier divisor on $Y$. 
Then $t(K_{Y}+B_{Y})+h^{*}A_{Y}$ is an ample $\mathbb{R}$-divisor on $Y$ for some $t\in (0,1)$. 
Taking a general member $A_{Y}\sim_{\mathbb{R}}\tfrac{1}{1-t}(t(K_{Y}+B_{Y})+h^{*}A_{Y})$, we have
$$K_{Y}+B_{Y}+h^{*}A\sim_{\mathbb{R}} (1-t)(K_{Y}+B_{Y}+A_{Y}).$$
Since $h$ is small, for each $m$ we have
$$H^{0}(X,\mathcal{O}_{X}(\llcorner mG \lrcorner))\simeq H^{0}(Y,\mathcal{O}_{Y}(\llcorner m h^{-1}_{*}G \lrcorner)).$$
Therefore, the graded ring $\mathcal{R}(X,G)$ is finitely generated if and only if 
$$\mathcal{R}(Y,h^{-1}_{*}G)=\underset{m \in \mathbb{Z}_{\geq 0}}{\bigoplus}H^{0}(Y,\mathcal{O}_{Y}(\llcorner m h^{-1}_{*}G \lrcorner))$$
is finitely generated. 
Furthermore, we have $h^{-1}_{*}G\sim_{\mathbb{R}}(1-t)(K_{Y}+B_{Y}+A_{Y})$. 
Applying \cite[Theorem 1.5]{hashizumehu} and \cite[Theorem 1.7]{hashizumehu}, we see that $\mathcal{R}(Y,h^{-1}_{*}G)$ is finitely generated, thus $\mathcal{R}(X,G)$ is finitely generated. 
We finish the proof. 
\end{proof}

\subsection{On the Kodaira type vanishing theorem}
We give an example which shows that the Kodaira type vanishing theorem does not necessarily hold in the situation of Theorem \ref{thm--abund-main-2} (\ref{thm--abund-main-2-(2)}) or Theorem \ref{thm--gen-non-vanishing-lp}. 
The following theorem is a stronger statement than Theorem \ref{thm--intro-van-example}.  

\begin{thm}\label{thm--van-example}
Let $d$ be a positive integer. 
Then, there is a $\mathbb{Q}$-factorial generalized klt pair $(X,B,\boldsymbol{\rm M})$, an ample $\mathbb{Q}$-divisor $A$ on $X$, and a Cartier divisor $D$ on $X$ such that 
\begin{itemize}
\item
${\rm dim}X=3$ and $K_{X}+B+A$ is pseudo-effective, 
\item
$D\sim_{\mathbb{Q}}K_{X}+B+A+t\boldsymbol{\rm M}_{X}$ for some $t>1$, and 
\item
${\rm dim}H^{1}(X,\mathcal{O}_{X}(D))\geq d$. 
\end{itemize}
\end{thm}

\begin{proof}
We prove the theorem in several steps. 

\begin{step5}\label{step1-example}
We construct some varieties and divisors. 

We put $V=\mathbb{P}^{1}$, and fix a very ample Cartier divisor $H_{V}$ such that $\mathcal{O}_{V}(H_{V})=\mathcal{O}_{V}(1)$.  
Let $W$ be an elliptic curve, and fix a very ample Cartier divisor $H_{W}$ on $W$. 
We set $p_{V}\colon V\times W \to V$ and $p_{W}\colon V\times W \to W$ as  projections. 
The divisor $p_{V}^{*}H_{V}+p_{W}^{*}H_{W}$ is a very ample Cartier divisor on $V\times W$. 
We construct a $\mathbb{P}^{1}$-bundle
$$Y:=\mathbb{P}_{V \times W}(\mathcal{O}_{V \times W}\oplus \mathcal{O}_{V \times W}(-p_{V}^{*}H_{V}-p_{W}^{*}H_{W}))\overset{f}{\longrightarrow} V\times W.$$
Let $T$ be the unique section corresponding to $\mathcal{O}_{Y}(1)$. 
It is obvious that $(Y,T)$ is a log smooth lc pair.
We define a Cartier divisor $L$ on $Y$ by $L=T+f^{*}(p_{V}^{*}H_{V}+p_{W}^{*}H_{W})$. 
By construction, we have 
$$K_{V\times W}=p_{V}^{*}K_{V}+p_{W}^{*}K_{W}\sim -2p_{V}^{*}H_{V}.$$ 
From this, we see that $L$ is basepoint free and the relation
$$K_{Y}+T+L\sim -2f^{*}p_{V}^{*}H_{V}$$
holds. 
Let $Y \to Z$ be the contraction induced by $L$. 
Since $L+f^{*}p_{W}^{*}H_{W}$ is basepoint free, we get a contraction $\pi \colon Y \to X$  induced by $L+f^{*}p_{W}^{*}H_{W}$. 
By construction, these morphisms are isomorphisms outside $T$. 
By calculations of intersection numbers of curves contracted by $\pi$, we see that all curves contracted by $\pi$ are also contracted by $Y \to Z$. 
Thus, the induced birational map $X\dashrightarrow Z$ is a morphism. 
We denote $Y\to Z$ by $\phi$ and $X \to Z$ by $\psi$. Then $L=\phi^{*}\phi_{*}L=\pi^{*}\psi^{*}\phi_{*}L$. 
Therefore, $\pi_{*}L$ is a nef and big $\mathbb{Q}$-Cartier divisor on $X$ and $L=\pi^{*}\pi_{*}L$. 
Since $L+f^{*}p_{W}^{*}H_{W}=\pi^{*}\pi_{*}(L+f^{*}p_{W}^{*}H_{W})$, the divisor $\pi_{*}f^{*}p_{W}^{*}H_{W}$ is $\mathbb{Q}$-Cartier and $f^{*}p_{W}^{*}H_{W}=\pi^{*}\pi_{*}f^{*}p_{W}^{*}H_{W}$. 
Since $H_{W}$ is very ample, $\pi_{*}f^{*}p_{W}^{*}H_{W}$ is semi-ample and $\pi_{*}f^{*}p_{W}^{*}H_{W}$ induces a contraction $g \colon X\to W$ such that $g\circ \pi =p_{W}\circ f$. 
We put $H_{X}=\pi_{*}f^{*}p_{V}H_{V}$. 
Since $L=\pi^{*}\pi_{*}L$ and $f^{*}p_{W}^{*}H_{W}=\pi^{*}\pi_{*}f^{*}p_{W}^{*}H_{W}$, we see that $H_{X}$ is $\mathbb{Q}$-Cartier and $\pi^{*}H_{X}=T+f^{*}p_{V}^{*}H_{V}$. 

By these discussion, we obtain the following diagram 
 $$
\xymatrix{
&&Y\ar[dl]_{f}\ar[dr]^{\pi}\\
&V\times W \ar[dl]^{p_{V}}\ar[dr]_{p_{W}}&&X\ar[dl]^{g}\\
V&&W
}
$$
and divisors $H_{V}$, $H_{W}$, $T$, $L=T+f^{*}(p_{V}^{*}H_{V}+p_{W}^{*}H_{W})$, and $H_{X}=\pi_{*}f^{*}p_{V}^{*}H_{V}$ such that 
\begin{enumerate}[(1)]
\item \label{proof--thm--van-example-(1)}
$(Y,T)$ is a log smooth lc pair, 
\item \label{proof--thm--van-example-(2)}
$H_{V}$ and $H_{W}$ are very ample and $K_{V \times W}\sim -2p_{V}^{*}H_{V}$. 
\item \label{proof--thm--van-example-(3)}
$K_{Y}+T+L+2f^{*}p_{V}^{*}H_{V}\sim 0$,
\item \label{proof--thm--van-example-(4)}
$\pi_{*}L$ is a nef and big $\mathbb{Q}$-Cartier divisor on $X$ such that $\pi^{*}\pi_{*}L=L$, and 
\item \label{proof--thm--van-example-(5)}
$H_{X}$ is $\mathbb{Q}$-Cartier and $\pi^{*}H_{X}=T+f^{*}p_{V}^{*}H_{V}$. 
\end{enumerate}
\end{step5}

\begin{step5}\label{step2-example}
In this step, we study properties of $X$ constructed in Step \ref{step1-example}. 
More precisely, we show that $X$ is $\mathbb{Q}$-factorial and $\pi^{*}K_{X}=K_{Y}-T$. 

Pick any Weil divisor $D$ on $X$. 
Then $\pi^{-1}_{*}D$ is linearly equivalent to the sum of a multiple of $T$ and the pullback of a Cartier divisor $G$ on $V\times W$. 
Since $V=\mathbb{P}^{1}$, we can find an integer $\alpha$ and a Cartier divisor $G_{W}$ on $W$ such that $G \sim \alpha p_{V}^{*}H_{V}+p_{W}^{*}G_{W}$. 
Thus, we have
$$D=\pi_{*}\pi^{-1}_{*}D \sim \pi_{*}f^{*}G \sim \alpha \pi_{*}f^{*}p_{V}^{*}H_{V}+\pi_{*}f^{*}p_{W}^{*}G_{W}=\alpha H_{X}+g^{*}G_{W}.$$ 
Since $H_{X}$ is $\mathbb{Q}$-Cartier, $D$ is also $\mathbb{Q}$-Cartier. 
Therefore, $X$ is $\mathbb{Q}$-factorial.  

By (\ref{proof--thm--van-example-(3)}) in Step \ref{step1-example}, we have $K_{X}+\pi_{*}L+2H_{X}\sim 0$. 
The negativity lemma shows 
$$K_{Y}+T+L+2f^{*}p_{V}^{*}H_{V}=\pi^{*}(K_{X}+\pi_{*}L+2H_{X}).$$
By using (\ref{proof--thm--van-example-(4)}) and (\ref{proof--thm--van-example-(5)}) in Step \ref{step1-example}, we obtain $\pi^{*}K_{X}=K_{Y}-T$. 
\end{step5}

\begin{step5}\label{step3-example}
In this step, we prove ${\rm dim}H^{2}(Y, \mathcal{O}_{Y}(-mK_{Y}+(m-1)T))=0$ for all positive integers $m$. 

We fix a positive integer $m$.  
Put $h=p_{W}\circ f \colon Y\to W$. 
Applying Relay spectral sequence to $\mathcal{O}_{Y}(-mK_{Y}+(m-1)T)$ and $h\colon Y\to W$, 
we see that it is sufficient to prove the following facts.
\begin{itemize}
\item $H^{2}\bigl(W,h_{*}\mathcal{O}_{Y}(-mK_{Y}+(m-1)T)\bigr)=0$, 
\item $H^{1}\bigl(W,R^{1}h_{*}\mathcal{O}_{Y}(-mK_{Y}+(m-1)T)\bigr)=0,$ and
\item $H^{0}\bigl(W,R^{2}h_{*}\mathcal{O}_{Y}(-mK_{Y}+(m-1)T)\bigr)=0$. 
\end{itemize}
The first fact is clear because ${\rm dim}W=1$. 
We also see that 
\begin{equation*}
\begin{split}
-mK_{Y}+(m-1)T=&K_{Y}-(m+1)K_{Y}+(m-1)T\\
\sim&K_{Y}+(m+1)(T+L+2f^{*}p_{V}^{*}H_{V})+(m-1)T \quad (\text{(\ref{proof--thm--van-example-(3)}) in Step \ref{step2-example}})\\
=&K_{Y}+2mT+(m+1)L+(2m+2)f^{*}p_{V}^{*}H_{V}\\
=&K_{Y}+(3m+1)L+2f^{*}p^{*}_{V}H_{V}-2mf^{*}p_{W}^{*}H_{W},
\end{split}
\end{equation*}
where the final relation follows from $L=T+f^{*}(p_{V}^{*}H_{V}+p_{W}^{*}H_{W})$. 
Since $H_{V}$ is very ample and $L$ is nef and big, we can apply relative Kawamata--Viehweg vanishing theorem to the sheaf $\mathcal{O}_{Y}(-mK_{Y}+(m-1)T)$ and the morphism $Y\to W$. 
Thus, we see that 
$$R^{1}h_{*}\mathcal{O}_{Y}(-mK_{Y}+(m-1)T)=R^{2}h_{*}\mathcal{O}_{Y}(-mK_{Y}+(m-1)T)=0.$$
From this, the second and the third facts stated above follow. 
In this way, we obtain the equality ${\rm dim}H^{2}(Y, \mathcal{O}_{Y}(-mK_{Y}+(m-1)T))=0,$ which is what we wanted to prove in this step.  
\end{step5}

\begin{step5}\label{step4-example}
In this step, we prove
$${\rm dim}H^{1}(X,\mathcal{O}_{X}(-lK_{X}))\geq {\rm dim}H^{0}(W,\mathcal{O}_{W}(2lH_{W}))$$
for every positive integer $l$ such that $lK_{X}$ is Cartier. 

Since $X$ is $\mathbb{Q}$-factorial (see Step \ref{step2-example}) and $T$ is the unique $\pi$-exceptional divisor, $-T$ is ample over $X$. 
Since $\pi^{*}K_{X}=K_{Y}-T$ (Step \ref{step2-example}), it follows that $Y$ is Fano over $X$. 
Thus, we have $R^{i}\pi_{*}\mathcal{O}_{Y}=0$ for every $i>0$. 
By Leray spectral sequence, it follows that 
$${\rm dim}H^{1}(X,\mathcal{O}_{X}(-lK_{X}))={\rm dim}H^{1}(Y,\mathcal{O}_{Y}(-l\pi^{*}K_{X}))={\rm dim}H^{1}(Y,\mathcal{O}_{Y}(-lK_{Y}+lT)).$$
We consider the exact sequence
$$0\longrightarrow \mathcal{O}_{Y}(-lK_{Y}+(l-1)T)\longrightarrow \mathcal{O}_{Y}(-lK_{Y}+lT)\longrightarrow \mathcal{O}_{T}((-lK_{Y}+lT)|_{T})\longrightarrow 0.$$
Since we have $H^{2}(Y, \mathcal{O}_{Y}(-lK_{Y}+(l-1)T))=0$ by Step \ref{step3-example}, taking the cohomology long exact sequence, we obtain the exact sequence
\begin{equation*}\tag{$*$}\label{proof--thm--van-example-(***)}
\begin{split}
H^{1}(Y, \mathcal{O}_{Y}(-lK_{Y}+lT))&\longrightarrow H^{1}(T,\mathcal{O}_{T}((-lK_{Y}+lT)|_{T}))\longrightarrow 0.
\end{split}
\end{equation*}
Since $(T+f^{*}(p_{V}^{*}H_{V}+p_{W}^{*}H_{W}))|_{T}\sim0$, we have 
$$(-lK_{Y}+lT)|_{T}\sim -lK_{T}-2lf^{*}(p_{V}^{*}H_{V}+p_{W}^{*}H_{W})|_{T}.$$
By identifying $T$ with $V\times W$ and using $K_{V \times W}\sim -2p_{V}^{*}H_{V}$, we obtain
$${\rm dim}H^{1}\bigl(T,\mathcal{O}_{T}((-lK_{Y}+lT)|_{T})\bigr)={\rm dim}H^{1}\bigl(V\times W,\mathcal{O}_{V \times W}(-2lp_{W}^{*}H_{W})\bigr).$$
Since $V=\mathbb{P}^{1}$, it is easy to check that $R^{i}p_{W*}\mathcal{O}_{V \times W}=0$ for every $i>0$. 
By applying Leray spectral sequence to $p_{W}\colon V \times W \to W$, we have
$${\rm dim}H^{1}\bigl(V\times W,\mathcal{O}_{V \times W}(-2lp_{W}^{*}H_{W})\bigr)={\rm dim}H^{1}(W,\mathcal{O}_{W}(-2lH_{W})).$$
Since $W$ is an elliptic curve, by Serre duality, we have
$${\rm dim}H^{1}(W,\mathcal{O}_{W}(-2lH_{W}))={\rm dim}H^{0}(W,\mathcal{O}_{W}(2lH_{W})).$$
Combining these relations, we obtain
\begin{equation*}
\begin{split}
{\rm dim}H^{1}(X,\mathcal{O}_{X}(-lK_{X}))=&{\rm dim}H^{1}(Y,\mathcal{O}_{Y}(-lK_{Y}+lT))\\
\geq&{\rm dim}H^{1}\bigl(T,\mathcal{O}_{T}((-lK_{Y}+lT)|_{T})\bigr)\qquad\quad (\text{by (\ref{proof--thm--van-example-(***)})})\\
=&{\rm dim}H^{1}\bigl(V\times W,\mathcal{O}_{V \times W}(-2lp_{W}^{*}H_{W})\bigr)\\
=&{\rm dim}H^{1}(W,\mathcal{O}_{W}(-2lH_{W}))\\
=&{\rm dim}H^{0}(W,\mathcal{O}_{W}(2lH_{W})).
\end{split}
\end{equation*} 
We finish this step.
\end{step5}

\begin{step5}\label{step5-example}
With this step we will finish the proof of Theorem \ref{thm--van-example}. 
In other words, for any positive integer $d$, we find a structure of generalized klt pair $(X,B,\boldsymbol{\rm M})$ on $X$, ample $\mathbb{Q}$-divisor $A$, and a Cartier divisor $D$ satisfying all the conditions of Theorem \ref{thm--van-example}. 

We fix a positive integer $d$. 
Pick any positive integer $l$ such that $lK_{X}$ is Cartier, $K_{X}+(l+1)\pi_{*}L$ is big, and ${\rm dim}H^{0}(W,\mathcal{O}_{W}(2lH_{W}))\geq d$. 
We can find such $l$ because $H_{W}$ is very ample and $\pi_{*}L$ is big ((\ref{proof--thm--van-example-(4)}) in Step \ref{step1-example}). 
By Step \ref{step2-example}, the pair $(X,0)$ is a $\mathbb{Q}$-factorial klt pair. 
Since $\pi_{*}L$ is nef and big ((\ref{proof--thm--van-example-(4)}) in Step \ref{step1-example}), we can find an effective $\mathbb{Q}$-divisor $B$ and an ample $\mathbb{Q}$-divisor $A$ on $X$ such that $(X,B+A)$ is klt and
$(l+1)\pi_{*}L\sim_{\mathbb{Q}}A+B.$
We define a b-nef $\mathbb{Q}$-b-Cartier $\mathbb{Q}$-b-divisor $\boldsymbol{\rm N}$ by the closure of $2f^{*}p_{V}^{*}H_{V}$. 
Since $(X,B+A)$ is $\mathbb{Q}$-factorial klt, we can find a rational number $u\in (0,1]$ such that the generalized pair $(X,B,u\boldsymbol{\rm N})$ is generalized klt. 
We put $\boldsymbol{\rm M}=u\boldsymbol{\rm N}$ and $D=-lK_{X}$. 

We show that $(X,B,\boldsymbol{\rm M})$, $A$ and $D$ satisfy all the conditions of Theorem \ref{thm--van-example}. 
Clearly ${\rm dim}X=3$, and $X$ is $\mathbb{Q}$-factorial by Step \ref{step2-example}. 
Since $K_{X}+A+B\sim_{\mathbb{Q}}K_{X}+(l+1)\pi_{*}L$, the divisor $K_{X}+B+A$ is big. 
Therefore, the first condition of Theorem \ref{thm--van-example} holds true. 
By (\ref{proof--thm--van-example-(3)}) in Step \ref{step1-example}, we have
$K_{X}+\pi_{*}L+\boldsymbol{\rm N}_{X}\sim 0$. 
Since $\boldsymbol{\rm M}=u\boldsymbol{\rm N}$ and $D=-lK_{X}$, we obtain
\begin{equation*}
\begin{split}
D=K_{X}-(l+1)K_{X}\sim_{\mathbb{Q}}K_{X}+B+A+\frac{l+1}{u}\boldsymbol{\rm M}_{X}
\end{split}
\end{equation*} 
Since $u\in (0,1]$, we have $\frac{l+1}{u}>1$. 
Thus, the second condition of Theorem \ref{thm--van-example} holds. 
Finally, Since $D=-lK_{X}$, by Step \ref{step4-example} and our choice of $l$, we have
$${\rm dim}H^{1}(X,\mathcal{O}_{X}(D))\geq {\rm dim}H^{0}(W,\mathcal{O}_{W}(2lH_{W}))\geq d.$$
Therefore, the third condition of Theorem \ref{thm--van-example} holds. 
\end{step5}
We complete the proof of Theorem \ref{thm--van-example}. 
\end{proof}

\begin{proof}[Proof of Theorem \ref{thm--intro-van-example}]
It is clear from Theorem \ref{thm--van-example}. 
\end{proof}

\section{Appendix. On non-$\mathbb{R}$-Cartier analogue of non-vanishing theorem}\label{sec5}

In this appendix, we give a small remark on the following non-$\mathbb{R}$-Cartier analogue of Theorem \ref{thm--abund-main-2} (\ref{thm--abund-main-2-(2)}).   

\begin{ques}[{cf.~Theorem \ref{thm--abund-main-2}}]\label{ques--gen-non-vanishing-3}
Let $(X,B,\boldsymbol{\rm M})$ be a projective generalized lc pair such that $B$ is a $\mathbb{Q}$-divisor and $\boldsymbol{\rm M}$ is a $\mathbb{Q}$-b-Cartier $\mathbb{Q}$-b-divisor. 
Let $A$ be an ample $\mathbb{Q}$-divisor on $X$. 
Suppose that $K_{X}+B+A+\boldsymbol{\rm M}_{X}$ is pseudo-effective. 

Then, is there an effective $\mathbb{Q}$-divisor $D_{\alpha}$ on $X$ such that $K_{X}+B+A+\alpha \boldsymbol{\rm M}_{X} \sim_{\mathbb{Q}}D_{\alpha}$ for every rational number $\alpha > 1$?
\end{ques}

Theorem \ref{thm--non-vanishing-lc-divisor-2} shows that the statement holds in the case where $K_{X}+B+A+(1-t)\boldsymbol{\rm M}_{X}$ is pseudo-effective in the sense of Definition \ref{defn--positivity-divisor} for a real number $t>0$. 
In particular, the statement holds when $\boldsymbol{\rm M}$ descends to a numerically trivial $\mathbb{Q}$-Cartier divisor. 
Indeed, if $\boldsymbol{\rm M}$ descends to a numerically trivial $\mathbb{Q}$-Cartier divisor, then there is a resolution $f\colon \tilde{X} \to X$ such that the divisor
$f^{*}(K_{X}+B+A+\boldsymbol{\rm M}_{X})-\boldsymbol{\rm M}_{\tilde{X}}$ is pseudo-effective, so $K_{X}+B+A$ is pseudo-effective in the sense of Definition \ref{defn--positivity-divisor}. 

If we can solve Question \ref{ques--gen-non-vanishing-3} affirmatively, there is an application to anti-nef canonical divisors of klt varieties.
The author learned the following question in a discussion with Gongyo.

\begin{ques}[Non-vanishing for nef anti-canonical divisors]\label{conj--non-van-antinef-klt}
Let $(X,0)$ be a projective klt pair.
Suppose that $-K_{X}$ is nef. 
In this situation, is there an effective $\mathbb{Q}$-divisor $D$ on $X$ such that $-K_{X}\sim_{\mathbb{Q}}D$?    
\end{ques}

\begin{thm}\label{thm--non-van-antinef-klt}
Assume that the statement of Question \ref{ques--gen-non-vanishing-3} holds for all generalized lc pairs and ample $\mathbb{Q}$-divisors. 
Then, Question \ref{conj--non-van-antinef-klt} can be solved affirmatively. 
\end{thm}

\begin{proof}
Assume that the statement of Question \ref{ques--gen-non-vanishing-3} holds for all generalized lc pairs and ample $\mathbb{Q}$-divisors. 
Let $(X,0)$ be a projective klt pair such that $-K_{X}$ is nef. 
We may assume that $-K_{X}$ is not big because otherwise the non-vanishing for $-K_{X}$ is obvious. 

Fix a very ample Cartier divisor $H$ on $X$, and consider the $\mathbb{P}^{1}$-bundle 
$$Y:=\mathbb{P}_{X}(\mathcal{O}_{X}\otimes \mathcal{O}_{X}(-H))\overset{f}{\longrightarrow} X.$$
We put $T$ as the unique section of $\mathcal{O}_{Y}(1)$. 
By construction, $(Y,T)$ is plt, $T+f^{*}H$ is basepoint free, and we have $K_{Y}+T+(T+f^{*}H)\sim K_{X}$. 
Let $\pi\colon Y \to Z$ be the contraction induced by $T+f^{*}H$. 
Then, $\pi$ is birational and $T$ is the unique $\pi$-exceptional divisor.  
We may find an ample Cartier divisor $A_{Z}$ on $Z$ such that $T+f^{*}H\sim \pi^{*}A_{Z}$. 
We define a $\mathbb{Q}$-b-divisor $\boldsymbol{\rm M}$ on $X$ by the closure of $-f^{*}K_{X}$. 
Since $K_{Y}+T+\pi^{*}A_{Z}+\boldsymbol{\rm M}_{Y}\sim_{\mathbb{Q}}0$ and $-K_{X}$ is nef, we see that 
$(Z,0,\boldsymbol{\rm M})$ is a generalized lc pair such that $K_{Z}+\boldsymbol{\rm M}_{Z}+A_{Z}\sim_{\mathbb{Q}}0$. 

By applying the statement of Question \ref{ques--gen-non-vanishing-3} to $(Z,0,\boldsymbol{\rm M})$ and $A_{Z}$, we can find an effective $\mathbb{Q}$-divisor $D_{Z}$ on $Z$ such that 
$$-\pi_{*}f^{*}K_{X}=\boldsymbol{\rm M}_{Z}\sim_{\mathbb{Q}}K_{Z}+A_{Z}+2\boldsymbol{\rm M}_{Z}\sim_{\mathbb{Q}}D_{Z}.$$
Pick an integer $m>0$ so that $-m\pi_{*}f^{*}K_{X} \sim mD_{Z}$. 
Since $T$ is the unique $\pi$-exceptional divisor, we can find a positive integer $n$ such that $H^{0}(Y, \mathcal{O}_{Y}(-mf^{*}K_{X}+nT))\neq 0$. 
For each integer $n'\leq n$, we consider the exact sequence
\begin{equation*}
\begin{split}
0\longrightarrow H^{0}(Y, \mathcal{O}_{Y}(-mf^{*}K_{X}+(n'-1)T)) &\longrightarrow  H^{0}(Y, \mathcal{O}_{Y}(-mf^{*}K_{X}+n'T))\\
& \longrightarrow H^{0}(T, \mathcal{O}_{T}((-mf^{*}K_{X}+n'T)|_{T}))
\end{split}
\end{equation*}
which is induced by the exact sequence
\begin{equation*}
\begin{split}
0\longrightarrow \mathcal{O}_{Y}(-mf^{*}K_{X}+(n'-1)T) &\longrightarrow \mathcal{O}_{Y}(-mf^{*}K_{X}+n'T)\\
& \longrightarrow \mathcal{O}_{T}((-mf^{*}K_{X}+n'T)|_{T}) \longrightarrow 0.
\end{split}
\end{equation*}
By using $(T+f^{*}H)|_{T}\sim 0$ and $T\simeq X$, we have
$$H^{0}(T, \mathcal{O}_{T}((-mf^{*}K_{X}+n'T)|_{T}))\simeq H^{0}(X, \mathcal{O}_{X}(-m K_{X}-n'H))$$
If $n'>0$, then $H^{0}(X, \mathcal{O}_{X}(-m K_{X}-n'H))=0$ because otherwise $-K_{X}$ is big which contradicts our assumption. 
Therefore, for every $0<n'\leq n$ we have
$$H^{0}(Y, \mathcal{O}_{Y}(-mf^{*}K_{X}+(n'-1)T)) \simeq  H^{0}(Y, \mathcal{O}_{Y}(-mf^{*}K_{X}+n'T)).$$
Since $H^{0}(Y, \mathcal{O}_{Y}(-mf^{*}K_{X}+nT))\neq 0$, by an descending induction on $n'$, we obtain
$$H^{0}(Y, -mf^{*}K_{X})\neq 0.$$ 
From this, we can find an effective $\mathbb{Q}$-divisor $D$ on $X$ such that $D\sim_{\mathbb{Q}}-K_{X}$. 
We finish the proof of Theorem \ref{thm--non-van-antinef-klt}. 
\end{proof}


\end{document}